\documentclass[11pt]{amsart}
\usepackage{amsmath,amsfonts,amssymb,amsthm, color}
\usepackage[usenames,dvipsnames,svgnames,table]{xcolor}
\usepackage{pgf}
\usepackage{tikz}
\usetikzlibrary{arrows,automata}

\advance\textwidth by 1.5cm 
\advance\oddsidemargin by -1.6cm
\advance\evensidemargin by -1.6cm

\usepackage[active]{srcltx}

\newcommand\del[1]{}

\newcommand\comadel[1]{}

\newcommand\delc[1]{}

\newcommand\deld[1]{}

\newcommand\dele[1]{}

\newcommand{\bd}{\begin{displaymath}}
\newcommand{\ed}{\end{displaymath}}

\newcommand{\norm}[1]{\left| \left| #1 \right|\right|}
\newcommand{\be}{\begin{eqnarray}}
\newcommand{\ee}{\end{eqnarray}}
\newcommand{\ben}{\begin{eqnarray*}}
\newcommand{\een}{\end{eqnarray*}}
\newcommand{\bs}{\begin{subequations}}
\newcommand{\es}{\end{subequations}}

\input{colordvi}

\usepackage{amsmath,amsthm,amssymb,epsfig,fullpage}
\usepackage{latexsym}

\newtheorem{theorem}{Theorem}[section]

\newtheorem{corollary}[theorem]{Corollary}
\newtheorem{lemma}[theorem]{Lemma}
\newtheorem{proposition}[theorem]{Proposition}

\newtheorem{example}[theorem]{Example}
\newtheorem{remark}[theorem]{Remark}

\numberwithin{equation}{section}

\def\R{\mathbb{R}}

\def\Nat{\mathbb{N}}

\def\L{\bf L}
\def\W\mathbf{W}

\def\L2{{L^2(\Omega)}}

\def\T{\mathcal{T}_{h}}

\def\veps{\varepsilon}

\def\norm#1#2{\v_hert\,#1\,\v_hert_{#2}}

\long\def\symbolfootnote[#1]#2{\begingroup%
\def\thefootnote{\fnsymbol{footnote}}\footnotetext[#1]{#2}\endgroup}

\newcommand{\Prob}{\mathbb{P}}
\newcommand{\E}{\mathbb{E}}

\newcommand{\cC}{C}
\newcommand{\cur}[1]{\left\{ #1 \right\}}
\newcommand{\sqa}[1]{\left[ #1 \right]}
\newcommand{\bra}[1]{\left( #1 \right)}
\newcommand{\abs}[1]{\left| #1 \right|}
\renewcommand{\norm}[1]{\left\| #1 \right\|}
\newcommand{\ang}[1]{\left< #1 \right>}
\renewcommand{\div}{\operatorname{div}}
\newcommand{\supp}{\operatorname{supp}}
\newcommand{\Td}{\mathbb{T}^d}
\newcommand{\cE}{\mathcal{E}}
\newcommand{\cL}{\mathcal{L}}
\newcommand{\cP}{\mathcal{P}}
\newcommand{\cW}{\mathcal{W}}
\newcommand{\conv}{\!*\!}
\newcommand{\w}{w}

\newcommand{\emp}{\mu}
\newcommand{\loc}{\rm{loc}}
\newcommand{\vphi}{\varphi}

\begin{document}

\title{A particle system approach to cell-cell adhesion models}
\author{Mikhail Neklyudov, Dario Trevisan}

\address{ Universit\`a degli Studi di Pisa,
 Dipartimento di Matematica,
 Largo Bruno Pontecorvo 5,
 56127 Pisa, Italy}
\email{dario.trevisan@unipi.it, misha.neklyudov@gmail.com}  
\date{\today}

\begin{abstract}
We investigate micro-to-macroscopic derivations in two models of living cells, in presence to cell-cell adhesive interactions. We rigorously address two PDE-based models, one featuring non-local terms and another purely local, as a a result of a law of large numbers for stochastic particle systems,  with moderate interactions in the sense of K.~Oelshchl\"ager~\cite{oelschlager_law_1985}.
\end{abstract}

\maketitle

\section{Introduction}

In mathematical biology, there is a vast and growing literature on modeling collective behavior of individuals (say, cells), moving and actively interacting between each other and with their environment, see e.g.\ the reviews~\cite{byrne_individual-based_2009, lowengrub_nonlinear_2010}. These are based on different mathematical tools, but mainly on (discrete) probabilistic individual based models or (continuous) partial differential equations (or even mixtures of the two). There are of course advantages from both sides, e.g.\ the continuous models are computationally more feasible, but discrete models are possibly richer in details. Usually, continuous models are seen as ``macroscopic'' limit of the discrete ``microscopic'' ones, but rigorous mathematical results can be very challenging~\cite{kipnis_scaling_1999}.


An interesting feature, common in some of the available models, is the presence of adhesive forces, which can describe of cell-cell communication, ultimately responsible for aggregate behaviors such as tissue formation in healthy organisms, but also invasion in case of cancer~\cite{armstrong_continuum_2006, deroulers_modeling_2009}. The main mathematical problem is then to reproduce (and motivate) the appearance of collective behaviors such as formation of clusters or more complex patterns, stemming from (limited range) interactions between individuals~\cite{morale_interacting_2005}. 

\vspace{1em}
In this paper, we address the specific problem of obtaining two such recently proposed PDE-based models from individual-based descriptions, via interacting stochastic differential equations on the $d$-dimensional torus $\Td$ (of course, $d \in \{1, 2, 3\}$ in realistic situations). The first model is of \emph{non-local} type, originally introduced in~\cite{armstrong_continuum_2006} and further developed in~\cite{painter_impact_2010, dyson_existence_2010, dyson_non-local_2013, painter_nonlocal_2015}, which in the simplified version that we consider (without extracellular matrix) reads as transport-diffusion PDE
\begin{equation}\label{eq:pde-non-local-intro}
\partial \rho + \div\bra{ \rho\,  b \conv [g(\rho)] } = \Delta \rho, \quad \text{on $[0,T] \times \Td$,}
\end{equation}
where the velocity field is given by the convolution between a fixed vector field $b: \Td \to \R^d$ and a non-linear function of the cell-density $\rho$, i.e., $g(\rho)$. This can be seen as a slight variant of a usual mean-field model, where the velocity would simply read as $b * \rho$: the role of $g$ is to avoid over-concentration, opposing to the aggregating action of the field $b$. 

The second model is purely \emph{local}, and we refer to it as introduced in~\cite{deroulers_modeling_2009}, where also a heuristic derivation via modified exclusion processes on a lattice is proposed. Instead of studying the explicit forms proposed therein (which we partially recover, see Example~\ref{ex:derouler}) we consider general PDE's of the form
\begin{equation} \label{eq:pde-local-intro} \partial_t \rho = \div(\rho \nabla u'(\rho)) + \Delta \rho. \quad \text{on $[0,T]\times \Td$,} \end{equation}
where $u(\rho)$ represents some internal energy, responsible for aggregation. In particular, we allow for non-convex $u$, as long as the total internal energy density $u(z) + z \log(z)$, $z\in (0,\infty)$, satisfies suitable convexity assumptions.

To obtain these two model, we introduce suitably systems of $n$ stochastic differential equations, and  let $n \to \infty$. It is well-known that qualitative differences may arise if the \emph{scaling regime} is such that each individual cell interacts, in average, with many neighbors (mean-field) or only with few ones (strong). In general, the case of interactions with few neighbors, although more realistic in the biological picture, is the one for which  mathematical theory is more challenging and few rigorous results are at disposal, starting from the seminal paper~\cite{varadhan_scaling_1991}.  The approach that we follow here consists in studying interaction whose strength lies ``in between'', also called moderate in~\cite{oelschlager_law_1985}, depending upon a parameter $\beta \in (0,1)$: where $\beta=0$ corresponds to the usual mean field and  $\beta=1$ would correspond to strong interactions. Unfortunately, our results are limited to $\beta \le \frac {d}{d+2}$ and such limitation seems very difficult to overcome in the case of  our main result concerning local models, see Theorem~\ref{thm:main-local}. In case of non-local models, Theorem~\ref{thm:main-non-local} has the same limitation but it seems plausible that some argument e.g.\ from \cite[Theorem 1]{oelschlager_law_1985} may allow to cover up to any $\beta <1$.

Let us briefly describe our proposed systems of particles. First, we fix a  smooth, compactly supported probability density $\w^1$, and introduce its rescaled versions (narrowly converging towards the Dirac measure at $0$)
\[ \w^n(x) = n^\beta \w^1(n^{\beta/d} x).\]
The role of $\w^n$ is to allow for evaluation of non-linear interactions between $n$ particles $(X^i_t)_{i=1}^n$: indeed, in general, it would not clear how to model non-linear functions of the empirical measure $\mu^n = \sum_{i=1}^n \delta_{X^i_t}$, but thanks to our choice of $\w^n$ (which can be interpreted as an ``effective'' profile of a single cell), we  introduce the smooth function
\[ \mu^n_t \conv w^n (x) = \frac 1 n\sum_{i=1}^n w^n(x-X^i_t),\]
and study the system (where $(B^i)_{i=1}^n$ are independent Brownian motions on $\Td$)
\[ dX^{i}_t = b_t * [g(\emp^n_t * \w^n )] (X^{i}_t) dt + \sqrt{2} d B^i_t, \quad \text{for $t \in[0,T]$,  and $i \in \cur{1, \ldots, n}$,}\]
to obtain in the limit~\eqref{eq:pde-non-local-intro}, and the system
\[ dX^{i}_t = - \nabla  \w^n * [u'(\mu^n_t\conv \w^n )] (X^{i}_t) dt + \sqrt{2} d B^i_t, \quad \text{for $t \in[0,T]$,  and $i \in \cur{1, \ldots, n}$,}\]
to find in the limit~\eqref{eq:pde-non-local-intro}. In the latter system, the drift term may appear odd at first, but it turns out that a similar expression for the drift was already introduced (for similar purposes, but essentially applied to mean-field type interactions) in~\cite{figalli_convergence_2008}, with $u(z)=z^m$, and it generalizes the original form for $p=2$ in~\cite{oelschlager_law_1985}. This expression has the ``right'' structure to allow for variational interpretations (as gradient flows in the space of probability measures, see Section~\ref{sec:local}). Intuitively, if we think  of $\w^n$ as the profile of a single cell, we may interpret the term $- \nabla  \w^n * [u'(\mu^n_t\conv \w^n )] = \w^n *[ - \nabla u'(\mu^n_t\conv \w^n )]$ as an average over such profile  of the velocity field associated to (minus) the gradient of $u'(\mu^n_t\conv \w^n)$.

On a technical side, our approach relies on suitable energy estimates, tightness and identification of the limit points (via application of uniqueness results). In this classical scheme for stability problems, the crucial element of novelty that we introduce, when compared with the moderate interaction limits in~\cite{oelschlager_law_1985} or subsequent developments~\cite{oelschlager_large_1990, morale_interacting_2005} is to look for energy estimates involving the Shannon entropy $\int_{\Td} \rho(x)\log \rho(x) dx $ in place of the $L^2$-norm $\int_{\Td} \rho(x)^2 dx$. The motivation of such substitution stems from recent approaches to non-linear PDE's as gradient flows in the space of probability measures~\cite{ambrosio_gradient_2008}, although here we do not employ sophisticated tools from that theory. It would be very interesting indeed to push forward our arguments, to deal with deterministic equations, instead of stochastic ones. Already in the ``easier'' non-local case, one would derive similar conclusion for interacting ODE's: when compared with other results in the literature of crowd motion, e.g.\ the very recent \cite{goatin_traffic_2015}, the problem here is that the non-linearity $g(\rho)$ acts before the convolution.

Besides the case of degenerate diffusions and the case of strong interactions $\beta=1$, the inclusion of growth of cells (via random duplication) and the extracellular matrix, as in~\cite{armstrong_continuum_2006}, are left open here.

After large parts of this work were completed, we became aware that the group \cite{buttenschon_poster_2015} is also proposing an independent and alternative microscopic derivation of the non-local model \eqref{eq:pde-non-local-intro}, using jump processes.
\vspace{1em}

Our paper is structured as follows: in Section~\ref{sec:general-results}, we introduce the notation and provide some general results on particle systems, following from straightforward applications of It\^o calculus; in Section~\ref{sec:non-local}, we study the non-local model~\eqref{eq:pde-non-local-intro} and in Section~\ref{sec:local} the model~\eqref{eq:pde-local-intro}; in Section~\ref{sec:appendix}, we prove some general results for which we were unable to find a quick reference, in particular for the tightness criterion Proposition~\ref{prop:tightness-general}.

\section*{Acknowledgments} Both authors would like to thank F.~Flandoli, C.~Olivera, M.~Leimbach and M.~Coghi for many useful discussions on this and related subjects. The second author is member of the GNAMPA group (INdAM).

\bibliographystyle{amsplain}

\section{Notation and basic results}\label{sec:general-results}

In this section, we introduce some notation and establish general inequalities regarding systems consisting of $n \ge 1$ (coupled) It\^o SDE's, taking values in $\Td = [-1/2, 1/2]^d$ (with periodic boundary), of the form
\begin{equation}\label{eq:sdes-general} dX^i_t = h_t(X^i) +  \sigma_t(X^i) dB^i, \quad \text{for $i \in \cur{1, \ldots, n}$, $t \in [0,T]$,}\end{equation}
where $h: \Omega \times [0,T]\times \Td \to \R^d$ and $\sigma: \Omega \times [0,T]\times \Td \to \R^{d \times d}$ are  progressively measurable and uniformly bounded maps and $(B^1, \ldots, B^n)$ are independent $\Td$-valued Brownian motions on a probability space $(\Omega, \mathcal{A}, \Prob)$, endowed with a filtration $(\mathcal{F}_t)_{t \in [0,T]}$, satisfying the usual assumptions.

In particular, we are interested in estimates for the stochastic process of empirical measures
\[ \mu = (\mu_t)_{t \in [0,T]} \in C([0,T];\cP(\Td)), \quad \text{where} \quad \mu_t :=\frac 1 n \sum_{i=1}^n \delta_{X^i_t},\]
and a mollified version of it, via a  function (fixed, in this section) $\w \in C^2(\Td; [0,\infty))$, such that $\int w =1$ and $w(x) = w(-x)$. 
 We let in this section 
\[ (\tilde{\mu}_t)_{t \in [0,T]} := (\mu_t \conv \w)_{t \in [0,T]} \in  C([0,T];\cP(\Td)).\]
Since $\tilde{\mu}_t$ has a (continuously) twice differentiable density with respect to the Lebesgue measure on $\Td$, we use the same notation for such density, which can be also expressed as
\begin{equation}\label{eq:density} \tilde{\mu}_t(x) = \frac 1 n \sum_{i=1}^n w(X^i_t-x), \quad \text{for every $x \in \Td$, $t \in [0,T]$.} \end{equation}
Let us notice that the assumption $\w \in C^2(\Td, [0,\infty))$ entails that  ${\w}^{1/2}$ is Lipschitz (a well-known fact, e.g. \cite[Lemma 3.2.3]{stroock_multidimensional_2006}) hence for some constant $c\ge 0$ one has
\begin{equation}\label{eqn:Wass_1}
|\nabla \w|^2(x)\leq c \w(x), \quad \text{for every $x \in \Td$.} 
\end{equation}
Moreover, $\w$ is bounded, hence $\tilde{\mu}$ is uniformly bounded on $[0,T]\times \Td$: in later sections, such bound will be seen to degenerate as $n \to \infty$, but in this section we use this fact only occasionally (and implicitly), to ensure that some objects are proper martingales (and not local ones).

%

Given $\vphi \in C^2(\Td)$, by applying It\^o formula to the composition of the function \[(x^1, \ldots, x^n)  \mapsto \frac{1}{n} \sum_{i=1}^n \vphi(x^i)\] with the $(\Td)^n$-valued process $(X^i, \ldots, X^n)$, it follows the a.s.\ identity, for every $t \in [0,T]$,
\begin{equation} \label{eq:martingale-problem-general} \int \vphi \mu_t - \int \vphi \mu_0 - \int_0^t  \sqa{ h_s \cdot \nabla \vphi + \frac 1 2 (\sigma_s \sigma^*_s) : \nabla^2 \vphi} \mu_s = \sum_{i=1}^n \int_0^t \nabla \vphi (X^i_t)  \cdot \sigma (X^i_t) dB_s^i,\end{equation}
and we recognize that the right hand side is a continuous martingale $M\vphi$, with quadratic variation process
\begin{equation}\label{eq:martingale-quadratic-variation-general}
[M\vphi]_t= \frac 1 {n^2} \sum_{i=1}^n \int_0^t \abs{ \sigma \nabla \vphi}^2(X^i_s) ds = \frac 1 n\int_0^t \int \abs{ \sigma \nabla \vphi }^2 \mu_s ds.
\end{equation}

In a more compact way, we may say that $\mu$ is a weak solution to the stochastic Fokker-Planck equation
\begin{equation}\label{eq:st-ce-general}
d \mu_t = \cL^* \mu_t dt + dM_t, \quad \text{on $(0,T)\times \T^d$,}
\end{equation}
in duality with $\vphi \in C^2(\Td)$, with the (random) Borel time-dependent Kolmogorov operator, defined on $C^2(\Td)$,
\begin{equation}\label{eq:kolmogorov} \cL_t \vphi (x) := h_t(x) \cdot \nabla \vphi(x) +  \frac 1 2 \sigma_t(x) \sigma^*_t(x): \nabla^2 \vphi(x), \quad \text{for $x\in \Td$, $t \in [0,T]$.}\end{equation}
To be more precise, we should specify that $M_t$ is a distributional-valued martingale, null at $t =0$, with quadratic variation process defined in~\eqref{eq:martingale-quadratic-variation-general}.

It\^o formula gives also the following result, for non-linear transformations of $\tilde{\mu}$.

\begin{proposition}[energy identity]\label{prop:energy-identity-general}
Let $(X^i, \ldots, X^n)$ be as in~\eqref{eq:sdes-general} and $F\in C^2(\R)$. Then, the process
\begin{equation}\label{eq:energy-identity}
 \int F(\tilde{\mu}_t(x)) dx - \int_0^t\int \sqa{ \cL_s ( F'(\tilde{\mu}_s)\conv \w )(y) + \frac {1} {2 n} \int F''(\tilde{\mu}_s(x)) \conv \abs{ \sigma_s(y) \nabla  \w(y-x)}^2 dx } \mu_s(dy) ds
\end{equation}
is a continuous martingale with quadratic variation process
\begin{equation}\label{eq:energy-identity-quadratic-varitaion} t \mapsto \frac 1 n\int_0^t \int \int F'(\tilde{\mu}_s(x) )\abs{ \sigma_s(y) \nabla  \w(y-x)}^2 dx \mu_s(dy) ds.\end{equation}
\end{proposition}

\begin{proof}
Although the proof can be seen as a straightforward application of It\^o formula to the continuously twice differentiable function
\[ (x^1, \ldots, x^n) \to \int F\bra{ \frac 1 n \sum_{i=1}^n \w(x^i - x) } dx, \]
we give a derivation from~\eqref{eq:st-ce-general}. Indeed, for every $x \in \Td$, letting $\vphi^x (y) := \w( x -y)$, the process 
\[ \tilde{\mu}_t(x) - \tilde{\mu}_0(x) - \int_0^t \int (\cL_s \vphi^x) \mu_s ds = (M\vphi^x)_t\]
is a continuous martingale, with quadratic variation obtained from~\eqref{eq:martingale-quadratic-variation-general}:
\[ [M\vphi^x]_t = \frac 1 n \int_0^t \int \abs{\sigma_s(y) \nabla \w(x-y)}^2 \mu_s(dy)ds.\]
By It\^o formula, the process
\[ F( \tilde{\mu}_t(x) ) - F( \tilde{\mu}_0(x)) - \int_0^t \sqa{ F'( \tilde{\mu}_s(x) ) \int (\cL_s \vphi^x) \mu_s }ds - \frac {1}{2n} \int F''( \tilde{\mu}_s(x) ) \int \abs{\sigma_s(y) \nabla \w(x-y)}^2 \mu_s(dy)\]
is a martingale, with quadratic variation process
\[ \frac 1 n \int_0^t F'(\tilde{\mu}_s(x)) \int \abs{\sigma_s(y) \nabla \w(x-y)}^2 \mu_s(dy) ds.\]
Integrating over $x \in \Td$, and exchanging integration with respect to $\mu_s$, the thesis follows, since for every bounded Borel function $G:\Td \to \R$, one has
\[\int  \int G(x) \cL_s \vphi^x(y) \mu_s(y) dx = \int \cL_s (G \conv \w)(y) \mu_s(dy),\]
and $x \mapsto F'(\tilde{\mu}_s(x))$ is bounded.
\end{proof}

\begin{remark}\label{rem:diffusion-constant}
In case $\sigma = \lambda Id$ is a constant multiple of the identity matrix,~\eqref{eq:energy-identity} reads as
\begin{equation}\label{eq:energy-identity-identity-matrix}
 \int F(\tilde{\mu}_t(x)) dx - \int_0^t\int \sqa{ \cL_s ( F'(\tilde{\mu}_s)\conv \w )(y) + \frac {\lambda^2} {2 n} \sqa{F''(\tilde{\mu}_s) } \conv \abs{ \nabla  \w}^2(y) } \mu_s(dy) ds
\end{equation}
and the quadratic variation process~\eqref{eq:energy-identity-quadratic-varitaion} is 
\[t \mapsto \frac {\lambda^2} n\int_0^t \int \sqa{F'(\tilde{\mu}_s)}\conv \abs{ \nabla  \w}^2 (y) \mu_s(dy) ds.\]
\end{remark}

From the energy identity~\eqref{eq:energy-identity}, we obtain suitable inequalities, such as the following one.

\begin{proposition}[energy inequality]\label{prop:energy-inequality}
Let $(X^i, \ldots, X^n)$ be as in~\eqref{eq:sdes-general}, with $\sigma = \lambda Id$ for some $\lambda > 0$ and $h \le c$ uniformly, for some constant $c>0$. Then, one has
\[
\sup_{t \in [0,T]} \E\sqa{ \int (\tilde{\mu}_t(x))^2 dx} + \frac{\lambda}{2} \E\sqa{ \int_0^T \int \abs{\nabla \tilde{\mu}_t}^2 dt} \le  2 \cur{\E\sqa{ \int (\tilde{\mu}_0(x))^2 dx} + T \frac {\lambda^2 \norm{ \nabla \w}_2^2} {n} } e^{2cT/\lambda}.
\]
%
\end{proposition}

\begin{proof}
We consider~\eqref{eq:energy-identity-identity-matrix} with $F(z) = z^2$, so that $F'(z) = 2z$, $F''(z) = 2$ and
\begin{equation*}\begin{split} \cL_s (F'(\tilde{\mu}_s \conv \w) (y) & = 2 h(y) \cdot \nabla (\tilde{\mu}_s \conv \w) (y) + \lambda \Delta (\tilde{\mu}_s \conv \w)(y) \\
& = 2 h(y) \cdot  \sqa{(\nabla\tilde{\mu}_s) \conv \w} (y) + \lambda (\Delta \tilde{\mu}_s)  \conv \w (y),
\end{split}
\end{equation*}
thus we estimate from above, for $s \in [0,T]$, and $\alpha>0$,
\begin{equation*}\begin{split}    \int [2 h_s(y) \cdot  \sqa{(\nabla\tilde{\mu}_s) \conv \w} (y) + & \lambda (\Delta \tilde{\mu}_s)  \conv \w (y)] \mu_s(dy) \le  \\ & \le 2c  \int \abs{(\nabla\tilde{\mu}_s) \conv \w} \mu_s - \lambda\int \abs{\nabla \tilde{\mu}_s}^2 \\
&\le 2 c \int \abs{\nabla\tilde{\mu}_s} \tilde{\mu}_s-\lambda \int \abs{\nabla \tilde{\mu}_s}^2 \\
&\le  \frac{2c}{\lambda}  \int (\tilde{\mu}_s)^2  -\frac \lambda 2  \int \abs{\nabla \tilde{\mu}_s}^2 \\
\end{split}.
\end{equation*}
where in the last inequality we split $2\abs{\nabla\tilde{\mu}_s} \tilde{\mu}_s \le \alpha(\tilde{\mu}_s)^2 + \alpha^{-1}\abs{\nabla\tilde{\mu}_s}^2$, for $\alpha = 2c/\lambda$. One also has
\[ \frac {\lambda^2}{2n} \int 2 \conv \abs{ \nabla  \w}^2(y) \mu_s(dy)  \le  \frac {\lambda^2 \norm{ \nabla \w}_2^2} {2n},\]
hence taking expectation we obtain, for $t \in [0,T]$,
\[ \E\sqa{ \int (\tilde{\mu}_t(x))^2 dx} + \frac{\lambda}{2} \E\sqa{ \int_0^t \int \abs{\nabla \tilde{\mu}_t}^2 dt} \le  \frac{2c}{\lambda} \E\sqa{ \int (\tilde{\mu}_s)^2}  + \E\sqa{ \int (\tilde{\mu}_0(x))^2 dx} + t\frac {\lambda^2 \norm{ \nabla \w}_2^2} {2n}, \]
and by Gronwall inequality we deduce the thesis.
\end{proof}

\section{A non-local model as limit of moderately interacting SDE's}\label{sec:non-local}

In this section, we study convergence as $n \to \infty$, for the empirical measures associated to the system of It\^o SDE's~\eqref{eq:sdes-general}, when we suitably choose both $\w=\w^n$ and $h=h^n$ depending upon $n$ (we also let $\sigma = \sqrt 2 Id$). We fix throughout $\beta >0$ with $\beta \le \frac{d}{d+2}$ and define for $n \ge 1$,
\begin{equation}\label{eq:wn}  \w^n(x) = n^\beta \w(n^{\beta/d} x), \quad \text{ for $x \in \Td$,} 
\end{equation}
 where
\begin{equation}\label{eq:w} \text{  $\w \in C^2(\R; [0,\infty))$ is supported on $(-1/2, 1/2)^d$, $\int_\R w =1$ and $w(x) = w(-x)$, for $x \in \R$.}
\end{equation} The definition is well-posed via the identification $\Td = [-1/2, 1/2]^d$ and one has $\int \w^n = 1$, $w^n(x) = w^n(-x)$  for every $x \in \Td$, for every $n \ge 1$. Moreover,~\eqref{eqn:Wass_1} reads as
\begin{equation}\label{eqn:Wass_2}
|\nabla \w^n|^2(x)\leq c n^{\beta(2/d+1)} \w^n(x), \quad \text{for $x \in \Td$,} 
\end{equation}
where $c>0$ is some absolute constant (not depending upon $n\ge 1$). The coefficient $n^\beta$ entails that $\w^n$ is a probability distribution on $\Td$ and, for $p \in [1, \infty]$,
\[ \norm{\w^n}_p = n^{\beta (p-1)/p} \norm{\w}_p.\]

Next, we let $h = h^n$ be random and depend upon the empirical law $\mu$ in the following way: we fix a Lipschitz function $g: [0,\infty) \to \R$ and a (possibly time dependent) uniformly bounded Borel vector field  $b: [0,T]\times \Td \to \R^d$ be and we let 
\[ h^n_t (x) = b_t \conv \sqa{ g(\mu_t \conv \w^n) } (x), \quad \text{for $t\in [0,T]$, $x \in\Td$.}\]
In a more rigorous formulation, we are interested in solutions to the system of SDE's 
\begin{equation}\label{eq:sdes-non-local} dX^{n,i}_t = b_t * [g(\emp^n_t * \w^n )] (X^{n,i}_t) dt + \sqrt{2} d B^i_t, \quad \text{for $t \in[0,T]$,  and $i \in \cur{1, \ldots, n}$,}\end{equation}
with $\mu_t^n = \frac 1 n \sum_{i=1}^n \delta_{X^{n,i}_t} \in \cP(\Td)$. Of course, for every $n\ge1$, there are no well-posedness issues, recalling identity~\eqref{eq:density}.

We are interested in the convergence of $\mu^n$ as $n \to \infty$ to weak solutions $\rho \in C([0,T]; \cP(\Td))$ to the non-linear and non-local PDE
\begin{equation}\label{eq:pde-non-local}
\partial \rho + \div\bra{ \rho b \conv [g(\rho)] } = \Delta \rho, \quad \text{on $[0,T] \times \Td$.}
\end{equation}
Solutions $\rho$ are understood in duality with functions $\varphi \in C^2(\Td)$, i.e.
Well-posedness for such equations is studied e.g.\ in~\cite{dyson_non-local_2013}: one has the following result. 

\begin{theorem}[well-posedness, non-local case]\label{thm:well-posedness-non-local}
For every $\bar{\rho} \in \cP(\Td) \cap L^2(\Td)$, there exists a unique weak solution $\rho \in C([0,T]; \cP(\Td)) \cap L^2([0,T]; W^{1,2}(\Td))$ to~\eqref{eq:pde-non-local}, with $\rho_0 = \bar{\rho}$.
\end{theorem}

\begin{proof}
The proof is based on a standard energy estimates, together with applications of Young convolution inequality, so we omit some details, to show that solutions are sufficiently smooth so that the following computation is rigorous. In particular, we focus on uniqueness (which is the part that we need for our study of convergence). Let $\rho$, $\tilde{\rho} \in C([0,T]; \cP(\Td)) \cap L^2([0,T]; W^{1,2}(\Td))$ be two solutions, and consider their difference $\rho- \tilde\rho$. Then,
\[ t \mapsto \frac 1 2\norm{\rho_t -\tilde \rho_t}_2^2 = \frac 1 2 \int (\rho_t(x)- \tilde \rho_t(x))^2 (x) dx\]
is absolutely continuous, with weak derivative
\begin{equation*}\begin{split} \partial_t\frac 1 2 \norm{\rho_t -\tilde\rho_t}_2^2 
& =  \int \nabla \bra{\rho_t- \tilde\rho_t} \cdot \sqa{ \rho_t b_t \conv [g(\rho_t)] - \tilde\rho b_t \conv [g(\tilde\rho_t)] } - \int  \abs{\nabla \bra{ \rho_t- \tilde\rho_t} }^2\\
& \le \frac 1 4 \int \abs{ \rho_t b_t \conv [g(\rho_t)] - \tilde{\rho} b_t \conv [g(\tilde{\rho}_t)] }^2 \\
& \le \frac { \norm{b_t \conv g(\rho_t) }_\infty}{4}  \norm{ \rho_t - \tilde{\rho}_t}^2_2 + \norm{\tilde{\rho}_t}_2^2 \norm{ b_t \conv \sqa{g(\rho_t)-g(\tilde{\rho}_t)} }^2_\infty \\
&\le  c^2 \norm{b_t}_\infty^2   \norm{ \rho_t - \tilde{\rho}_t}^2_2  + \bra{ \operatorname{Lip}g}^2  \norm{\tilde{\rho}_t}_2^2 \norm{ b_t}_\infty^2  \norm{ \rho_t - \tilde{\rho}_t }^2_1\\
 &\le  \sqa{ c^2   +  \bra{ \operatorname{Lip}g}^2  \norm{\tilde{\rho}_t}_2^2 } \norm{b_t}_\infty^2\norm{\rho_t - \rho_t^2 }_2^2,
\end{split}\end{equation*}
where $c \ge 0$ is some constant such that $g(z) \le c(1+ z)$, for $z \in [0,\infty)$. By Gronwall lemma
\[ \sup_{t \in [0,T]} \norm{\rho_t - \tilde \rho_t}_2^2 \le \exp\cur{ \bra{T c^2   + \bra{\operatorname{Lip}g }^2  \norm{\tilde{\rho}}_{L^2_t(L^2_x)}^2 } \norm{b}_{L^\infty_t(L^\infty_x)}^2 }\norm{\rho_0 - \tilde \rho_0}_2^2,\]
from which uniqueness follows.
\end{proof}

\begin{theorem}[convergence, non-local case]\label{thm:main-non-local}
Fix $0<\beta \le \frac {d}{d+2}$ and let $\w$, $g$, $b$ be as above. For $n \ge 1$, let  $\w^n$ as in~\eqref{eqn:Wass_2} and let $(X^n_t)_{t\in [0,T]}$ be a $(\Td)^n$-valued process satisfying~\eqref{eq:sdes-non-local}, with $\mu^n_t:= \frac 1 n \sum_{i=1}^n \delta_{X^{n,i}_t} \in \cP(\Td)$.

If the random variables $(\mu^n_0)_n \subseteq \cP(\Td)$ converge in law towards some (random) $\bar{\mu} \in \cP(\Td)$ , with
\begin{equation}
\label{eq:well-prepared-initial-datum-non-local}
 \limsup_{n \to \infty } \E\sqa{ \int (\mu^n_0 \conv \w^n (x))^2 dx} < \infty, \end{equation}
then $(\mu^n)_{n} \in C([0,T]; \cP(\Td))$ converge in law towards the (uniquely determined in law) random variable $\mu \in C([0,T]; \cP(\Td))$  such that $\mu_0 = \bar{\mu}$ and is a.s.\ concentrated on the distributional solutions to~\eqref{eq:pde-non-local} in the class $C([0,T]; \Td) \cap L^2([0,T]; W^{1,2}(\Td))$.
\end{theorem}

\begin{corollary}[i.i.d.\ initial data, non-local case]\label{coro:stability-independent-non-local}
Fix $0<\beta \le \frac {d}{d+2}$ and let $\w$, $g$, $b$ be as above.  For $n \ge 1$, let  $\w^n$ as in~\eqref{eqn:Wass_2} and let $(X^n_t)_{t\in [0,T]}$ be a $(\Td)^n$-valued process satisfying~\eqref{eq:sdes-non-local}, with $\mu^n_t:= \frac 1 n \sum_{i=1}^n \delta_{X^{n,i}_t} \in \cP(\Td)$, where $(X^{n,i}_0)_{i =1}^n$ are independent, uniformly distributed random variables with law $\bar{\mu} = \bar{\rho}(x) dx \in \cP(\Td) \cap L^2(\Td)$.

Then, $(\mu^n)_{n\ge 1} \in C([0,T]; \cP(\Td))$ converge in probability to the solution $\mu \in C([0,T]; \cP(\Td)) \cap L^2([0,T]; W^{1,2}(\Td))$ to~\eqref{eq:pde-non-local},  with $\mu_0  = \bar{\mu}$.
\end{corollary}

\begin{proof}[Proof of Corollary~\ref{coro:stability-independent-non-local}]
By Lemma~\ref{lem:iid-general}, the sequence $\mu_0^n * \w^n$ satisfies~\eqref{eq:well-prepared-initial-datum-non-local} and it converges in law towards $\bar{\mu}$, which is deterministic. Hence, Theorem~\ref{thm:main-non-local} entails that $\mu^n$ converge in law towards the unique solution to~\eqref{eq:pde-non-local} described in Theorem~\ref{thm:well-posedness-non-local}, with $\mu_0  = \bar{\mu}$. It is then well-known that convergence in law towards a deterministic random variable self-improves to convergence in probability.
\end{proof}

\begin{proof}[Proof of Theorem~\ref{thm:main-non-local}]
The proof follows a standard scheme: first, we show tightness of the sequence of the laws $(\mu^n)_n$; then, we prove that any limit point of  $(\mu^n)_n$ is concentrated on solutions to~\eqref{eq:pde-non-local} for which Theorem~\ref{thm:well-posedness-non-local} applies, hence the sequence in fact converges and the limit is uniquely identified. 

{\noindent \it Step 1} (tightness){\it .} Given the general results in Section~\ref{sec:general-results}, tightness for the law of $(\mu^n)$ follows from Proposition~\ref{prop:tightness-general}, with $h=h^n$ and $\sigma^n = \sqrt 2 Id$, and choosing e.g.\ $c_1= 2$, $c_2=4$. Indeed, the vector field $h^n$ is uniformly bounded, since
\begin{equation}\label{eq:uniformbound-vector-field}
\norm{b_t \conv [g(\tilde{\mu}^n_t)] }_\infty  \le \norm{b_t}_\infty \norm{g(\tilde{\mu}^n_t) }_1  \le c \norm{b_t}_\infty \sqa{ 1 + \norm{\tilde{\mu}^n_t}_1 } = 2 c \norm{b_t}_\infty,\end{equation}
where we use the fact that $g$ is Lipschitz, hence for some constant $c>0$, one has $g(z) \le c(1+ z)$ for $z \in [0,\infty)$. Therefore, the right hand side of~\eqref{eq:general-tightness} is uniformly bounded as $n \to \infty$:
\[ \E\sqa{ \int_0^T \int \bra{|h_t|^{2} + |\sigma|^{4}} d\mu_t dt}  \le \int_0^T 2 c \norm{b_t}_\infty dt   + 4T \]

Up to extracting a subsequence $n_k \to \infty$, we may assume that $\mu^n$ converges in law towards some random variable $\mu$, with values in $C([0,T]; \cP(\Td))$.

{\noindent \it Step 2} (limit){\it .} By Proposition~\ref{prop:energy-inequality}, we have a uniform estimate for $\tilde{\mu}$ in $L^2([0,T]; W^{1,2}(\Td))$, namely
\[ \E\sqa{ \int_0^T \int \abs{\nabla \tilde{\mu}^n_t}^2 dt} \le  2\sqrt{2} \cur{\E\sqa{ \int (\tilde{\mu}^n_0(x))^2 dx} + T \frac {2\norm{ \nabla \w^n}_2^2} {n} } e^{2 \sqrt{2} c \norm{b_t}_\infty T},\]
where we use once again the uniform bound~\eqref{eq:uniformbound-vector-field}, but also~\eqref{eq:well-prepared-initial-datum-non-local} and crucially~\eqref{eqn:Wass_2}, which entails
\[T \frac{ 2\norm{ \nabla \w^n}_2^2}{n} \le T c n^{ \beta(2/d+1) -1} \norm{ \w^n}_1^2 \le  T c n^{ \beta(2/d+1) -1}, \]
hence a uniform bound as $n \to \infty$, because $\beta \le d/(d+2)$.

We are in a position to apply Proposition~\ref{prop:weak-strong}: as a first consequence, the limit random variable $\mu$ admits the representation  $\mu_t(dx) = \rho_t(x)dx$, with $\rho \in L^2([0,T]; W^{1,2}(\Td))$. As a second consequence, we show that $\mu$ is concentrated on weak solutions to~\eqref{eq:pde-non-local}. Indeed, given $\varphi \in C^{2}(\Td)$, $t \in [0,T]$,  we  pass to the limit, as $n \to \infty$, in the identity between (real valued) random variables
\[\int \varphi d \mu_t^n - \int \varphi d \mu_0^n - \int_0^t\int  \sqa{ \bra{\nabla \varphi} \cdot b_s \conv [g(\tilde{\mu}^n_s)]  + (\Delta \varphi) }\mu_s^n  ds = (M^n\varphi)_t,\]
which is the specialization of~\eqref{eq:martingale-problem-general} to this case and $M^n\varphi$ is a martingale null at $0$ with quadratic variation ~\eqref{eq:martingale-quadratic-variation-general}, which reads as
\[ [M^n\varphi]_t = \frac 2 n \int_0^t \int \abs{\nabla \varphi}^2\mu^n_s ds.\]
To obtain in the limit
\[ \int \varphi d \mu_t  - \int \varphi d \mu_0  - \int_0^t\int  \sqa{ \bra{\nabla \varphi} \cdot b_s \conv [g(\rho_s)]  + (\Delta \varphi) }d\mu_s^n  ds = 0,\]
we notice first that the quadratic variation above entails $(M^n\varphi)_t \to 0$ strongly in $L^2(\Prob)$, hence we may focus on the remaining terms. The key remark is that the functional defined on $C([0,T]; \cP(\Td) ) \times L^2( [0,T]; L^2(\Td) )$,
\begin{equation}\label{eq:functional-limit-non-local} (\nu, r) \mapsto \int \varphi d \nu_t - \int \varphi d \nu_0 - \int_0^t\int  \sqa{ \bra{\nabla \varphi} \cdot b_s \conv [g( r_s )]  + (\Delta \varphi) } \nu_s ds\end{equation}
satisfies all the assumptions of Proposition~\ref{prop:weak-strong}: continuity with respect to both variables in $C([0,T]; \cP(\Td) ) \times L^2( [0,T]; L^2(\Td) )$ follows trivially for all terms, except possibly
\[ (\nu, r)  \to  \int   \bra{\nabla \varphi} \cdot b_s \conv [g( r_s )] \nu_s.\]
If $\nu^m \to \nu$ in $C([0,T]; \cP(\Td) )$ and $r^m \to r$ in $L^2( [0,T]; L^2(\Td) )$, we estimate 
\[ \int \bra{\nabla \varphi} \cdot b_s \conv [g( r_s^m )] \nu_s^m -  \int \bra{\nabla \varphi} \cdot b_s \conv [g( r_s )] \nu_s\]
by adding and subtracting $\int \bra{\nabla \varphi} \cdot b_s \conv [g( r_s )] \nu_s^m$, so that
\begin{equation}\begin{split}\label{eq:uniform-continuity} \abs{ \int \bra{\nabla \varphi} \cdot b_s \conv [g( r_s^m )] \nu_s^m - \int \bra{\nabla \varphi} \cdot b_s \conv [g( r_s )] \nu_s^m} & \le \norm{ \nabla \varphi}_{\infty} \norm{ b_s \conv \bra{ g(r_s) - g(r_s^m )} }_\infty\\
& \le \norm{ \nabla \varphi}_{\infty} \norm{b_s}_\infty \operatorname{Lip}(g) \norm{r_s - r_s^m }_2 \to 0
\end{split}\end{equation}
and
\[\int \bra{\nabla \varphi} \cdot b_s \conv [g( r_s )] \nu_s^m - \int \bra{\nabla \varphi} \cdot b_s \conv [g( r_s )] \nu_s \to 0 \]
because the convolution is bounded and continuous. Uniform continuity for $F(\nu, \cdot)$ follows from the same argument which gives~\eqref{eq:uniform-continuity}.

By Proposition~\ref{prop:weak-strong} applied e.g.\ to the composition of~\eqref{eq:functional-limit-non-local} with the absolute value function, we deduce that the identity
\[ \E\sqa{ \abs{\int \varphi d \mu^n_t - \int \varphi d \mu_0^n - \int_0^t\int  \sqa{ \bra{\nabla \varphi} \cdot b_s \conv [g( \tilde{\mu}^n_s )]  + (\Delta \varphi) } \mu_s^n ds } }= \E\sqa{\abs{ (M^n\varphi)_t}} \]
converges as $n \to \infty$ towards
\[\E\sqa{ \abs{\int \varphi d \mu_t - \int \varphi d \mu_0 - \int_0^t\int  \sqa{ \bra{\nabla \varphi} \cdot b_s \conv [g( \tilde{\mu}_s )]  + (\Delta \varphi) } \mu_s ds } } = 0. \] 
Since $\varphi \in C^2(\Td)$ and $t \in [0,T]$ are arbitrary, by a standard density argument we deduce that $\mu_t(dx) = \rho_t(x)dx$ is concentrated on weak solutions to~\eqref{eq:pde-non-local} belonging to $C([0,T]; \cP(\Td)) \cap L^2([0,T]; W^{1,2}(\Td))$, for which Theorem~\ref{thm:well-posedness-non-local} applies.
\end{proof}

\section{A local model as a limit of moderately interacting SDE's}\label{sec:local}

In this section, we study convergence for the empirical measures of the system of It\^o SDE's~\eqref{eq:sdes-general} when $n\to\infty$ and we choose a (sufficiently smooth) $u: [0,\infty) \to \R$ and let
\[ h^n(x) := - \nabla  \w^n * [u'( \tilde{\emp}^n_t)](x), \quad \text{for $x \in \Td$,}\]
where $\w = \w^n$ as in~\eqref{eq:wn}, for some fixed $\beta \in (0,\frac{d}{d+2}]$ and using the notation $\tilde{\mu}^n_t = \mu^n_t \conv \w^n$. We also let $\sigma = \sqrt{2} Id$, as in the previous section.

Explicitly, the system of SDE's reads as
\begin{equation}\label{eq:sdes-local} dX^{n,i}_t = - \nabla  \w^n * [u'(\mu^n_t\conv \w^n )] (X^{n,i}_t) dt + \sqrt{2} d B^i_t, \quad \text{for $t \in[0,T]$,  and $i \in \cur{1, \ldots, n}$,}\end{equation}

Interaction energies of a similar form appear in~\cite{figalli_convergence_2008}, although our main result, Theorem~\ref{thm:main-local} is different in spirit, since it deals with ``moderate interactions'' and ``adhesive'' forces. It could be regarded as a generalization of \cite[Theorem 2]{oelschlager_law_1985} to different types of energies, although our statement does not cover directly that case.

The stochastic Fokker-Planck equation~\eqref{eq:kolmogorov} for $(\mu^n_t)_{t\in [0,T]}$ reads as
\begin{equation}\label{eq:weak-sn-local} d\emp^n_t = \sqa{ \div ( \emp^n_t \nabla \w^n * [u'(\tilde{\emp}^n_t )] ) +  \Delta \emp^n_t} dt + dM_t, \quad \text{on $[0,T]\times \Td$,}\end{equation}
in duality with functions in $C^2(\Td)$, where the quadratic variation of the distributional-valued martingale $M$ is given as in~\eqref{eq:martingale-quadratic-variation-general}.

As $n \to \infty$, since the quadratic variation is infinitesimal, we expect $\emp^n_t(dx) \to \rho_t(x) dx$ in the space $C([0,T]; \cP(\Td))$, where $\rho$ solves the non-linear PDE
\begin{equation}\label{eq:pde-limit} \partial_t \rho = \div(\rho \nabla u'(\rho)) + \Delta \rho. \quad \text{on $[0,T]\times \Td$.} \end{equation}
Theorem~\ref{thm:main-local} provides a rigorous justification of this fact, under suitable assumptions on $u$. Our derivation ultimately relies on the interplay between equivalent formulations of~\eqref{eq:pde-limit}: that of purely diffusion-type
\begin{equation}\label{eq:pde-diffusion} \partial_t \rho = \Delta P(\rho),\end{equation}
where we introduced the ``pressure'' $P(z) = zu'(z) - u(z) +z$, and that of transport-type
\begin{equation}\label{eq:pde-transport} \partial_t \rho = \div ( \rho \nabla F'(\rho) ),\end{equation}
where we introduced the ``internal energy'' $F(z) = u(z) + z \log z$. The formal equivalence between the two can be seen by straightforward calculus. In the latter form~\eqref{eq:pde-transport},  we have at our disposal on more recent uniqueness results, via the theory of gradient flows in the space of probability measures $\cP(\Td)$, rigorously developed in~\cite{ambrosio_gradient_2008}. Indeed, it can be interpreted as the gradient flow of the energy $\cE: \cP(\Td) \to [0,\infty]$, given by
\begin{equation} \label{eq:energy-limit} \cE( \mu ) = \begin{cases} \int u(\rho(x) ) + \rho(x) \log \rho(x) dx &  \text{if $\mu(dx) = \rho(x) dx$,}\\
+\infty & \text{otherwise,}\end{cases}\end{equation}
with respect to the Riemannian-like metric induced by the optimal transport distance on $\cP(\Td)$ with respect to the cost given by the distance squared, i.e.\
\begin{equation}\label{eq:transport} d(\mu, \nu) := \inf_{ \eta \in \Gamma(\mu,\nu) } \bra{ \int_{\Td \times \Td} |x-y|^2 \eta(dx,xy)}^{1/2}\end{equation}
where $\Gamma(\mu, \nu)$ is the set of probability measures $\eta$ on $\Td \times \Td$ with given marginals $(\mu, \nu)$, the relaxed ``transport plans'' in the Kantorovich sense.


A well-posedness result originating from this interpretation is described in \cite[Section 10.4.3]{ambrosio_gradient_2008}: existence of a weak formulation of the gradient flow~\eqref{eq:pde-transport} is ensured if
\begin{equation}\label{eq:F-existence} \begin{split}\text{$F: [0,\infty) \to \R$ is convex, differentiable in $(0,\infty)$, with $F(0) =0$,}\\
\text{ $\lim_{z\to \infty} \frac{F(z)}{z} = \infty$\quad  \quad and \quad \quad $\lim_{z \to 0^+} \frac{F(z)}{z^\alpha} >-\infty$, for some $\alpha > \frac{d}{d+2}$,}
\end{split}\end{equation}
so in particular, the internal energy~\eqref{eq:energy-limit} is lower semicontinuous. Uniqueness then holds if moreover
\begin{equation}\label{eq:F-uniqueness}\text{ $z \mapsto z^{d} F(z^{-d})$ is convex and non increasing on $(0,\infty)$.}\end{equation}

Indeed, \cite[Theorem 11.2.5]{ambrosio_gradient_2008} gives the following well-posedness result (which is even more than what it useful for our present purposes).

\begin{theorem}[well-posedness, local-case]\label{thm:well-posedness}
Let $F$ satisfy~\eqref{eq:F-existence} and~\eqref{eq:F-uniqueness}, set $P(z) := z F'(z) - F(z)$. Then, for every $\bar{\mu} \in \cP(\Td)$, there exists a unique distributional solution to~\eqref{eq:pde-diffusion} among which satisfy
\[ (\mu_t)_{t\in [0,T]} \in C([0,T]; \cP(\Td) ) \cap AC^2_{\loc}((0,T]; \cP(\Td)), \quad \mu_0 = \bar{\mu}\]
and, for every $t \in (0,T]$, one has 	$\mu_t(dx) =\rho_t(x) dx$, with
\[ P(\rho) \in L^1_{\loc}((0,T]; W^{1,1}_{\loc} (\Td) ), \quad  \int_{\Td} \frac{ \abs{\nabla P( \rho_t ) (x) }^2}{ \rho_t(x) }dx  \in  L^1_{\loc}((0,T]).\]
\end{theorem}

The strength of Theorem~\ref{thm:well-posedness} is that it can be applied directly to distributional solutions of diffusion-type~\eqref{eq:pde-diffusion}, i.e.\ those which $\mu_t(dx) =\rho_t(x) dx$, for $t\in (0,T]$, with $P(\rho) \in L^1_{\loc}((0,T];L^1_{\loc} (\Td) )$ and
\[\int_0^T \int \sqa{\partial_t \varphi(t,x) +\Delta \varphi(t,x)} P(\rho_t(x) )dx dt= 0, \quad \text{for every $\varphi \in C^2_c( (0,T)\times \Td )$.}\]

Our target is the case where $u$ may represent an adhesive force, i.e.\ it is not necessary convex, although the total internal energy is still assumed to be convex. Actually, we introduce  the following assumption on $F$, describing the fact that the energy $u$ is controlled by the entropy, i.e.\ the internal energy associated to the Brownian motion:
\begin{equation}\label{eq:F-assumption-2} \text{	there exists $\lambda<1$ such that $
\abs{z u''(z)}\le \lambda  $, for every $z \in (0,\infty)$.}\end{equation}
This condition entails that $z \mapsto F(z)$ is convex, and that $F(z) \le c( z \log z + z + 1)$, for some constant $c>0$ depending upon $\lambda$ only. In particular, for many purposes, we may replace $F(z)$ with $z \log z$.

%

\begin{theorem}[convergence, local case] \label{thm:main-local}
Let $u :[0,\infty) \to \R$, with $u\in C^2[0,\infty)$, define $F(z) = u(z) + z \log z $ and assume that~\eqref{eq:F-existence},~\eqref{eq:F-uniqueness} and~\eqref{eq:F-assumption-2} hold. Fix $0 < \beta \le \frac d {d+2}$, let $\w$ satisfy~\eqref{eq:w} and define $\w^n$, for $n \ge 1$, as in~\eqref{eq:wn}. For $n \ge 1$, let $(X^n_t)_{t\in [0,T]}$ be a $(\Td)^n$-valued process satisfying~\eqref{eq:sdes-local}, with $\mu^n_t:= \frac 1 n \sum_{i=1}^n \delta_{X^{n,i}_t} \in \cP(\Td)$.

If the random variables $(\mu^n_0)_n \subseteq \cP(\Td)$ converge in law towards some $\bar{\mu} \in \cP(\Td)$, with
\[
 \limsup_{n \to \infty } \E\sqa{ \int F (\mu^n_0 \conv \w^n (x)) dx} < \infty, \]
then the random variables $(\mu^n)_{n} \in C([0,T]; \cP(\Td))$ converge in law towards the (uniquely determined in law) random variable $\mu \in C([0,T]; \cP(\Td))$  such that $\mu_0 = \bar{\mu}$ and a.s.\ concentrated on the distributional solutions to~\eqref{eq:pde-diffusion} in the class of Theorem~\ref{thm:well-posedness}.
\end{theorem}

The proof of the following corollary goes along the same lines as Corollary~\ref{coro:stability-independent}.

\begin{corollary}[i.i.d.\ initial data, local case]\label{coro:stability-independent}
Let $u :[0,\infty) \to \R$, define $F(z) = u(z) + z \log z$ and assume that~\eqref{eq:F-existence},~\eqref{eq:F-uniqueness} and ~\eqref{eq:F-assumption-2}  hold. Fix $0 < \beta \le \frac d {d+2}$, let $\w$ satisfy~\eqref{eq:w} and define $\w^n$, for $n \ge 1$, as in~\eqref{eq:wn}. For $n \ge 1$, let $(X^n_t)_{t\in [0,T]}$ be a $(\Td)^n$-valued process satisfying~\eqref{eq:sdes-local}, with $\mu^n_t:= \frac 1 n \sum_{i=1}^n \delta_{X^{n,i}_t} \in \cP(\Td)$, where $(X^{n,i}_0)_{i =1}^n$ are independent, uniformly distributed random variables with law $\bar{\mu} = \bar{\rho}(x) dx \in \cP(\Td) \cap L^2(\Td)$.

Then, the random variables $(\mu^n)_{n} \in C([0,T]; \cP(\Td))$ converge in probability towards the unique solution to~\eqref{eq:pde-diffusion} described in Theorem~\ref{thm:well-posedness}, with $\mu_0  = \bar{\mu}$.
\end{corollary}

\begin{example}\label{ex:derouler}
Our investigation is motivated by some PDE's obtained as heuristic limits of discrete models in~\cite{deroulers_modeling_2009}, based on variants of exclusion processes. The energy $u$ therein is a polynomial and they study only densities which are a-priori uniformly bounded by some constant (indeed, in their model, for large densities, the diffusion coefficient becomes negative). Therefore, to recover similar energies, we may consider e.g.\ $u$ such that $u''(z) = c (1-z)^+$, for some $c \in \R$. It is not difficult to check that, if $|c|$ is small enough (depending also on the dimension $d$), conditions~\eqref{eq:F-existence},~\eqref{eq:F-uniqueness} and~\eqref{eq:F-assumption-2} are satisfied by $F(z) = u(z) + z \log z$, hence our result applies.
\end{example}

\begin{remark}
It is evident that our assumption \eqref{eq:F-assumption-2}  excludes the cases of $u$ being a convex polynomial, such as in \cite{oelschlager_law_1985} or \cite{figalli_convergence_2008}. In fact, under such an assumption, it seems possible to slightly modify our proof of Theorem \ref{thm:main-local} to show existence of a limiting law for the sequence $(\mu^n)_{n\ge1}$, concentrated on weak (distributional) solutions to \eqref{eq:pde-diffusion}, but it is presently not clear whether these solutions will have enough regularity so that Theorem \ref{thm:well-posedness} applies. 
\end{remark}



As for Theorem~\ref{thm:main-non-local}, the proof is in two steps, corresponding here respectively to Proposition~\ref{prop:energy} and Proposition~\ref{prop:limit}: first, we show tightness of the of the laws of $(\mu^n)_n$, then, we prove that any limit point of  $(\mu^n)_n$ is concentrated on solutions to~\eqref{eq:pde-diffusion} for which Theorem~\ref{thm:well-posedness} applies, hence the sequence must converge, since the limit is unique. 

The main idea is to investigate the following ``approximation'' of the energy~\eqref{eq:energy-limit},
\[ \cE^n( \mu ) := \int F ( \mu \conv \w^n (x)) dx = \cE(\mu \conv \w^n), \quad \text{for $\mu \in \cP(\Td)$.}\]
Jensen inequality entails $\cE^n (\mu) \le \cE( \mu)$; moreover $\lim_{n \to \infty} \cE^n(\mu) = \cE(\mu)$, since $\mu \mapsto \cE(\mu)$ is lower semicontinuous. However,~\eqref{eq:weak-sn-local} is not the gradient flow of $\cE^n$ with respect to the transport distance~\eqref{eq:transport}: indeed, the equation for the gradient flow of $\cE^n$ reads as
\[ \partial_t \mu_t = \div ( \mu_t  \nabla \w^n \conv [u'(\mu_t \conv \w^n )+ \log(\mu_t \conv \w^n) ] )\]
 differs from~\eqref{eq:weak-sn-local} in two aspects: the absence of martingales (it is deterministic) and the expression
 \[ \div(\mu_t  \nabla \w^n \conv\sqa{\log( \mu_t \conv \w^n)} ) = \div\bra{\mu_t   \w^n \conv \sqa{ \frac{\nabla \mu_t \conv \w^n}{\mu_t \conv \w^n}} }.\]
  in place of $\Delta \mu_t$. 
  
Nevertheless, we are able to deduce an approximate version of the so-called energy dissipation identity (see~\cite{ambrosio_gradient_2008}) for~\eqref{eq:weak-sn-local}, involving the ``squared norm'' of the gradient of $\cE^n$, i.e.\
\[ | \nabla \cE^n|^2 (\mu) := \int \abs{\nabla \w^n \conv F'( \mu \conv \w^n) }^2 \mu, \]
and the Fisher information
\[\mathcal{I}( \mu ) = \begin{cases} 4 \int | \nabla \sqrt{ \rho (x) } |^2 dx &  \text{if $\mu(dx) = \rho(x) dx$,}\\
+\infty & \text{otherwise.}\end{cases}\]

Moreover, assumption~\eqref{eq:F-assumption-2} entails that $| \nabla \cE^n|^2 (\mu) \le c \mathcal{I}( \tilde{\mu} )$, where $c$ is some absolute constant, since
\begin{equation}\label{eq:fisher-info-inequality}\begin{split}  \int  \abs{ \nabla \w^n \conv  [u'(\tilde \emp^n)]}^2 \emp^n & =  \int  \abs{ \w^n \conv \sqa{u''(\tilde{\emp}^n) \nabla \tilde \emp^n}}^2 \emp^n
 \le  \int  \w^n \conv  \abs{ {u''(\tilde{\emp}^n) \nabla \tilde \emp^n}}^2 \emp^n \\
 & \le  \int \abs{ {u''(\tilde{\emp}^n) \nabla \tilde \emp^n}}^2 \w^n \conv \emp^n =\int \abs{u''(\tilde{\emp}^n) \tilde{\emp}^n}^2 \frac { \abs{ \nabla \tilde \emp^n}^2}{\tilde \emp^n }  \\
 &  \le 4\lambda^2 \int \abs{\nabla \sqrt{\tilde{\mu}^n} }^2  = \lambda^2\mathcal{I} (\tilde{\mu}^n),
\end{split}\end{equation}
so that we may focus on the entropy and Fisher information terms only.

\begin{proposition}[energy dissipation and tightness]\label{prop:energy}
Under the assumptions of Theorem~\ref{thm:main-local}, there exists some constant $c>0$ (independent of $n\ge 1$) such that, for every $n\ge 1$, one has
\begin{equation}\label{eq:energy-estimate-full}  \sup_{t\in [0, T] } \E\sqa{\cE^n (\mu_t^n)} + \E\sqa{\int_0^T \sqa{ \cE^n\bra{\mu^n_t } + \mathcal{I}\bra{\tilde{\mu}^n_t } } dt }\le  c \bra{ \E\sqa{\cE^n (\mu_0^n) }+1}. \end{equation}
Moreover, the sequence of laws of $\mu^n$ is tight in $C([0,T]; \cP(\Td))$. 
\end{proposition}

\begin{proof}
We apply It\^o formula to the entropy process $t\mapsto \operatorname{Ent}(\tilde{\mu}^n_t) = \int \tilde\emp ^n_t\log(\tilde{ \emp}^n_t)$ (to be rigorous, we use the approximation $\int \tilde\emp ^n_t\log(\tilde{ \emp}^n_t+\veps)$ and then let $\veps \downarrow 0$), and arguing as in Proposition~\ref{prop:energy-identity-general} and Remark~\ref{rem:diffusion-constant}, we obtain that it can be rewritten as the sum of a finite variation process, with time derivative
\[  - \int\ang{ \nabla \w^n \conv  \log (\tilde \emp^n_t), \nabla \w^n \conv [u'(\tilde \emp^n_t )] } \emp^n  -\int \frac{ |\nabla \tilde \emp^n_t |^2} {\tilde \emp^n_t}+ \frac 1 n \int \frac{ \emp^n_t * |\nabla\w^n |^2} {\tilde \emp^n_t} \]
and a martingale, whose quadratic variation process has time derivative given by
\[  \frac 2 n  \int \abs{ \log(\tilde \emp ^n_t) \conv \nabla \w^n}^2 \mu_t^n. \]
The crucial point is to bound from above the quantity (we omit to specify $t \in [0,T]$ for brevity)
\[ - \int \ang{ \nabla \w^n \conv  \log(\tilde{\emp}^n), \nabla \w^n \conv [u'(\tilde \emp^n )] } \emp^n  \le \frac 1 2 \int \abs{ \nabla \w^n \conv  [u'(\tilde \emp^n)]}^2  \emp^n  + \frac 1 2 \int \abs{ \nabla \w^n \conv  \log (\tilde \emp^n)}^2  \emp^n
\]
where we splitted $|\ang{a,b}| \le \frac{a^2}2 + \frac{b^2}2$. By~\eqref{eq:fisher-info-inequality} and a similar argument with $\log$ in place of $u'$, we obtain
\[ - \int \ang{ \nabla \w^n \conv  \log(\tilde{\emp}^n), \nabla \w^n \conv [u'(\tilde \emp^n )] } \emp^n\le \frac {\lambda^2 +1}{2} \mathcal{I}(\tilde{\emp}^n).\]

 To bound from above the term
\[\frac 1 n \int  \frac{ \emp^n_t * |\nabla\w^n |^2}{\tilde{\emp}^n_t}\]
we use~\eqref{eqn:Wass_2}, to deduce
\[ \frac 1 n \int  \frac{ \emp^n_t * |\nabla\w^n |^2}{\tilde{\emp}^n_t} \le c n^{\beta(2/d+1) -1}  \int \frac{ \emp^n_t * \w^n }{\tilde{\emp}^n_t}\le  c n^{\beta(2/d+1) -1}.  \]
Taking expectation, so that the martingale term gives no contribution, we have the inequality, for every $t \in [0,T]$,
\[ \E\sqa{ \operatorname{Ent}(\tilde{\mu}_t^n) } \le \E\sqa{\operatorname{Ent}(\tilde{\mu}_0^n) }- \frac {1 -\lambda^2}{2} \E\sqa{ \int_0^t  \mathcal{I}(\tilde{\mu}^n_s) ds } + c n^{\beta(2/d+1) -1},\]
hence 
~\eqref{eq:energy-estimate-full} follows, since $\beta \le \frac{d}{d+2}$ and by assumption~\eqref{eq:F-assumption-2}, it is equivalent to the  $\operatorname{Ent}(\tilde{\mu})$ or $\cE^n(\mu)$,  and $| \nabla \cE^n|^2 (\mu) \le \lambda^2 \mathcal{I}( \tilde{\mu} )$.
 
The last statement, about the tightness for the laws of $\mu^n$, follows from Proposition~\ref{prop:tightness-general}, with $c_1 =2$ and any choice of $c_2>2$. Indeed, it is sufficient to notice that the diffusion coefficients are uniformly bounded, with $\abs{\sigma} = \sqrt{2}$ and that the energy inequality~\eqref{eq:energy-estimate-full} and~\eqref{eq:fisher-info-inequality} entails an integral bound for the drift terms. 
\end{proof}

\begin{proposition}[limit]\label{prop:limit}
Under the assumptions of Theorem~\ref{thm:main-local}, any limit point  of the laws of $\mu^n$, as $n \to \infty$, is a probability measure concentrated on weak solutions $\mu  \in AC^2( [0,T]; \cP(\R^d) )$ with $\mu_t(dx) = \rho_t (x) dx$ for every $t \in [0,T]$, solving~\eqref{eq:pde-diffusion}, in duality with $f \in C^2_b(\Td)$, with
\begin{equation}\label{eq:regularity-estimates-rho} (\rho_t)_{t \in [0,T]} \in L^1([0,T]; W^{1,1}(\Td)) \quad \text{and} \quad \int_0^T \int \frac{ |\nabla P(\rho )|^2 }{\rho} < \infty.\end{equation}
\end{proposition}

\begin{proof}
By the previous proposition, we may consider a converging subsequence $\mu^{n_k}$; to keep notation simple, we omit to write the subscript $k$, and write only $n$ below. Given $\varphi \in C^2_b(\Td)$, $t \in [0,T]$, we  pass to the limit in law, as $n \to \infty$, in the identity between (real valued) random variables
\begin{equation}\label{eq:approximate-local-weak} \int \varphi d \mu_t^n - \int \varphi d \mu_0^n = \int_0^t\int  \sqa{-\ang{ \nabla \varphi, \nabla \w^n \conv u'(\tilde{\mu}^n_s) } + (\Delta \varphi) }d\mu_s^n  ds + M^n_t \varphi,\end{equation}
to obtain
\[ \int \varphi d \mu_t  - \int \varphi d \mu_0  = \int_0^t \int (\Delta \varphi)(x) P( \rho_s(x)) dx ds.\]
The key point is to apply Proposition~\ref{prop:weak-strong} with exponent $p =1$, to the functional defined on $C([0,T]; \cP(\Td) ) \times L^1( [0,T]; L^1(\Td) )$ by
\begin{equation}\label{eq:functional-limit-local} (\nu, r) \mapsto \int \varphi d \nu_t - \int \varphi d \nu_0 - \int_0^t\int  \bra{ \Delta \varphi} P( r_s )ds.\end{equation}
Being $z \mapsto P(z)$ Lipschitz (its derivative is $z u''(z) +1$), continuity for~\eqref{eq:functional-limit-local} with respect to both variables in $C([0,T]; \cP(\Td) ) \times L^1( [0,T]; L^1(\Td) )$ is trivial. We are in a position to apply Proposition~\eqref{prop:weak-strong} since the bound on the Fisher information~\eqref{eq:energy-estimate-full} entails an $L^1$ bound on $\nabla \tilde{\mu}_t^n$, via the estimate  
\begin{equation}\label{eq:l1-norm-nabla-r} \int \abs{\nabla r } = \int \abs{\nabla \bra{\sqrt{r}}^2} =  2 \int \sqrt{r} \abs{\nabla \sqrt{r}} \le  1 + \mathcal{I}(r), \end{equation}
for any sufficiently smooth probability density $r(x)dx \in \cP(\Td)$.

By Proposition~\ref{prop:weak-strong} applied e.g.\ to the composition of~\eqref{eq:functional-limit-non-local} with the absolute value function, we deduce that the limit point $\mu$ is concentrated on absolutely continuous measures, $\mu_t (dx)= \rho_t(x) dx$, with $\nabla \rho \in L^1([0,T]; L^1(\Td))$ and that
\[ \E\sqa{ \abs{\int \varphi d \mu^n_t - \int \varphi d \mu_0^n - \int_0^t\int  \bra{\Delta \varphi} P(\tilde{\mu}^n_s) ds } } \to \E\sqa{ \abs{\int \varphi d \mu_t - \int \varphi d \mu_0 - \int_0^t\int  \bra{\Delta \varphi} P(\rho_s) ds } }\]
as $n \to \infty$. On the other side, by~\eqref{eq:approximate-local-weak}, for every $n\ge 1$, the term of the sequence above coincides with
\[ \E\sqa{ \abs{ \int_0^t  \sqa{ \int \bra{\Delta \varphi} P(\tilde{\mu}^n_s) + \int  \sqa{\ang{ \nabla \varphi, \nabla \w^n \conv u'(\tilde{\mu}^n_s) } - (\Delta \varphi) }\mu_s^n  }ds   - M_t^n \varphi }}\]

The random variable $M^n_t \varphi$ converges to $0$ strongly in $L^2(\Prob)$ (it is sufficient to use the isometry for martingales and the quadratic variation~\eqref{eq:martingale-quadratic-variation-general}). We decompose $P(z) =P_u(z)+ z$,  where $P_u(z) = z u'(z) - u(z)$, and we immediately estimate
\[ \E\sqa{ \abs{ \int_0^t  \int \bra{\Delta \varphi} \sqa{ \tilde{\mu}^n_s - \mu^n_s} ds } }\to 0,
 \]
 as $n \to \infty$ (e.g., again by Proposition~\ref{prop:weak-strong}). Hence, we have to deal only with the terms
\[  \E\sqa{ \abs{ \int_0^t  \sqa{ \int \bra{\Delta \varphi} P_u(\tilde{\mu}^n_s) + \int  \sqa{\ang{ \nabla \varphi, \nabla \w^n \conv u'(\tilde{\mu}^n_s) } } \mu^n_s } ds } },\]
where the difficulty arises because it involves non-linear transformations of $\tilde{\mu}$. To show that also this contribution is infinitesimal, we introduce, for $s \in [0,t]$, the commutator between the ``derivation'' $\ang{\nabla \varphi, \nabla \cdot}$ and the convolution operator $\w^n  * $, i.e.\ 
\[ \cC^n_s := \int \ang{\nabla \varphi,  \nabla \w^n \conv u'(\tilde{\mu}^n_s)}\mu^n_s - \int \ang{\nabla \varphi, \nabla u'(\tilde{\mu}^n_s)} \tilde \mu^n_s .\]
For simplicity, we omit to specify $s \in [0,t]$ in what follows, and we use the identity
\begin{equation*}
\begin{split}
\int  \ang{ \nabla \varphi, \nabla \w^n \conv u'(\tilde{\mu}^n) } d\mu^n 
 & = \cC^n  + \int \ang{ \nabla \varphi, \nabla  u'(\tilde{\mu}^n) } \tilde \mu^n \\
  & = \cC^n - \int \ang{ \nabla \varphi, \nabla  P_u(\tilde{\mu}^n) }\\
  & = \cC^n - \int \bra{\Delta \varphi} P_u(\tilde{\mu}^n).
\end{split}
\end{equation*} 
The thesis therefore amounts to the fact that $\cC^n$ is infinitesimal. In turn, this can be seen as follows:
\begin{equation*}\begin{split} | \cC^n | 
& = \abs{\int \ang{\nabla \vphi,  \nabla \w^n * u'(\tilde{\mu}^n_t)}\mu^n_t - \int \ang{\nabla \vphi, \nabla u'(\tilde{\mu}^n_t)} \tilde \mu^n_t}\\
& = \abs{ \int \int w^n(y) \ang{ \bra{\nabla \vphi (x)  - \nabla\vphi (x-y)}  \nabla u'(\tilde{\mu}^n)(x-y) } dy \mu^n(dx)}\\
& \le \norm{ \nabla ^2 \vphi }_{L^\infty(\Td)} \norm{y}_{L^\infty(\supp \w^n)} \int\sqa{ w^n \conv \abs{  \nabla u'(\tilde{\mu}^n)} } \mu^n \\
& = \norm{ \nabla ^2 \vphi }_{L^\infty(\Td)} \norm{y}_{L^\infty(\supp \w^n)} \int \abs{  \nabla u'(\tilde{\mu}^n)} \tilde{\mu}^n \to 0,
\end{split}\end{equation*}
where the inequality above follows from writing $\nabla \vphi(x) - \nabla \vphi(x-y) = \int_0^1 \nabla^2 \vphi(x - \veps y) y d\veps$. The quantities in the last line above are infinitesimal as $n \to \infty$, since
\[  \norm{y}_{L^\infty(\supp \w^n)}\le n^{-\beta/d}\norm{y}_{L^\infty(\supp \w^1)} \quad \text{and} \quad \int \abs{  \nabla u'(\tilde{\mu}^n)} \tilde{\mu}^n \le \lambda \int \abs{ \nabla \tilde{\mu}^n }, \]
and the integral is uniformly bounded, by~\eqref{eq:l1-norm-nabla-r}.

Next, we show that~\eqref{eq:regularity-estimates-rho} holds. Indeed, this follows from the fact that $\sqrt{\tilde{\mu}^n}$ is bounded in $L^2( \Omega \times [0,T]; W^{1,2}(\Td))$ and (up to our choice of a subsequence) it converges weakly towards $\sqrt{\rho}$: hence we have $\sqrt{\rho} \in L^2( \Omega \times [0,T]; W^{1,2}(\Td))$, which entails $\rho \in L^1(\Omega \times [0,T]; W^{1,1}(\Td))$. Moreover, we have
\[ \int_0^T \int \frac{ \abs{ \nabla P(\rho_t) }^2 }{\rho_t} dt \le \int_0^T \int \bra{\rho_t u''(\rho_t) + 1}^2 \frac{ \abs{ \nabla \rho_t }^2 }{\rho} \le {(\lambda+1)^2} \int_0^T \mathcal{I}(\rho_t)dt,\]
hence~\eqref{eq:regularity-estimates-rho} follows.

%

Finally, to show that $\mu  \in AC^2( [0,T]; \cP(\Td) )$, we notice that, for $s$, $t \in [0,T]$, with $s \le t$, we can write
\[ \int \varphi \rho_t -  \int \varphi \rho_s = \int_s^t \ang{ \nabla \varphi, \frac{\nabla P(\rho_r)}{\rho_r} } \rho_r  =\int_s^t \ang{ \nabla \varphi, v_r } \rho_r, \]
i.e., the curve is solution of the transport formulation~\eqref{eq:pde-transport}, hence it is absolutely continuous, with metric speed bounded from above by
\[ \int \abs{v }^2 \rho  = \int \frac{ \abs{ \nabla P(\rho)}^2}{\rho} < \infty.\]
\end{proof}

\section{Auxiliary results} \label{sec:appendix}

\begin{lemma}[moment bound]\label{lem:iid-general}
Let $\beta \in [0,1]$, let $\bar{\rho}(x) dx$, $\w^1(x)dx \in \cP(\Td)$ and, for $n \ge 1$, define $\w^n(x) = n^\beta \w^1(n^{\beta/d} x)$. Let $(X^{i})_{i =1}^n$ be independent, uniformly distributed random variables with common law $\bar{\rho}(x) dx$. If $\mu^n =\frac 1 n \sum_{i=1}^n \delta_{X^i}$, one has
\[ \E\sqa{ \int \bra{\mu^n \conv \w^n }^2(x) dx } \le  \norm
{\w^1}_\infty + \int \bar{\rho}^2(x) dx.\]
\end{lemma}

\begin{proof}
Exchanging expectation and integration with respect to $x \in \Td$, one has
\begin{equation*}\begin{split} \E\sqa{ \int \bra{\mu^n \conv \w^n }^2(x) dx } & = \int \E\sqa{ \bra{\mu^n \conv \w^n }^2(x) }dx  \\
\text{(by definition of $\mu^n$)}\quad & = \int \E\sqa{ \bra{ \frac 1 n \sum_{i=1}^n \w^n(X^i-x) }^2 }dx\\
 & = \int   \frac 1 {n^2} \sum_{i,j=1}^n\E\sqa{ \w^n(X^i-x) \w^n(X^j-x)} dx \\
\text{(by independence)}   \quad & = \int   \sqa{ \frac {n-1} {n}  (\bar{\rho} \conv \w^n)^2(x) + \frac 1 n \bar{\rho} \conv (\w^n)^2 (x) }dx.
\end{split}\end{equation*}
Since $\w^n \le \norm{\w^1}_\infty n^\beta$, we have
\[ \int   \frac 1 n \bar{\rho} \conv (\w^n)^2 (x) dx \le \norm{\w^1}_\infty n^{\beta-1} \int \bar{\rho} \conv \w^n (x) dx = \norm{\w^1}_\infty n^{\beta-1}.\]
By Jensen inequality, $(\bar{\rho} \conv \w^n)^2 \le \bar{\rho}^2 \conv \w^n$, hence
\[ \E\sqa{ \int \bra{\mu^n \conv \w^n }^2(x) dx } \le  \frac{n-1}{n}\int \bar{\rho}^2(x) dx + \norm
{\w^1}_\infty n^{\beta-1} \]
and the thesis follows.
\end{proof}

\begin{proposition}[tightness criterion for $\mu$]\label{prop:tightness-general}
For every $T \in\R$, $T>0$, $c_1, c_2 \in \R$, with $c_1 \ge 2$, $c_2 >2$ and $d \in \Nat \setminus \cur{0}$, there exists a coercive functional 
\[ \Psi: C([0,T]; \cP(\Td) ) \to [0, \infty]\]
such that, for every $n \ge 1$ and $\R^{n \times d}$-dimensional process $X = (X^{i})_{i=1}^n$ on $t \in [0,T]$ satisfying
\[ dX^i_t = h_t(X^i) +  \sigma_t(X^i) dB^i, \quad \text{for $i \in \cur{1, \ldots, n}$, $t \in [0,T]$,}\]
letting $\mu := \frac 1 n \sum_{i=1}^n \delta_{X^i} \in C([0,T];\cP(\Td))$, it holds
\begin{equation}\label{eq:general-tightness} \E\sqa{ \Psi( \mu ) } \le \E\sqa{ \int_0^T \int \bra{|h_t|^{c_1} + |\sigma_t|^{c_2 }} \mu_t dt}.\end{equation}
\end{proposition}

The crucial aspects of the result above are that $\Psi$ does not depend upon $n\ge1$ and it allows for $c_1=2$ (while $c_2$ must be strictly larger than $2$). In the proof, we argue similarly as in the classical Levy's modulus of continuity for Brownian motion, but splitting between the absolutely continuous  and the martingale parts, and using Burkholder-Gundy inequalities for Hilbert-space valued martingales $M$ (see e.g.~\cite{kotelenez_continuity_1984}) in the form
\begin{equation} \label{eq:bg-hilbert} \E\sqa{ \sup_{t \in [0,T] } \| M_t \|^{p}  } \le c_p \E\sqa{ [M_T]^{p/2}  }, \quad \text{for $p \in [2,\infty)$,}\end{equation}
where the constant $c_p$ depends on $p$ only. We say that a functional $\Psi$ on a metric space $X$, taking non negative values, is coercive if its sublevels $\cur{\Psi \le c }$ are compact, for every $c \ge 0$. A quantitative formulation for the tightness for the law of a random variable $\mu$ with values in $X$ follows then by an inequality for $\E\sqa{ \Psi(\mu) }$, via Markov inequality.

\begin{proof}
To simplify notation, we prove the thesis for $T=1$ only. By Ascoli-Arzel\`a theorem, the result amounts to provide estimates on the tightness of $\mu_t$, for every $t \in [0,1]$ as well as estimates on the modulus of continuity of $t \mapsto \mu_t$, e.g.\ with respect to the distance $\cW_2$.  Since $\Td$ is compact, tightness of $\mu_t$  is obvious, hence we focus on the modulus of continuity. By definition of $\cW_2$, we estimate from above the distance in terms of the coupling induced by the process $X$, i.e.\
\begin{equation}\label{eq:wasserstein-coupling} d( \mu_s, \mu_t ) \le \bra{ \frac 1 n \sum_{i=1}^n \abs{ X^{i}_s - X^{i}_t}^2 }^{1/2},\end{equation}
hence it is sufficient to estimate the modulus of continuity for the $\R^{n \times d}$-valued process $X^{n} /\sqrt{ n} $, with respect to the Euclidean distance. We provide a detailed derivation as follows.

Let $D([0,1]^2) \subseteq  C([0,1]^2 ; [0,\infty))$ be the set of (possibly degenerate) continuous distance functions on $[0,1]^2$, i.e.\ continuous functions $\gamma$ such that $\gamma(s,t) = \gamma(t,s)$ and
\[ \gamma(s,t) \le  \gamma(s,r) +  \gamma(r,t), \quad \text{for every $r$,$s$ $t \in [0,1]$.}\]
Such a set is closed in $C([0,1]^2 ; [0,\infty))$, endowed with uniform convergence. Moreover, since $\gamma(0,0) =0$, compactness in $D([0,1]^2)$ follows uniquely from uniform estimates on the modulus of continuity.

Given $\mu \in C([0,1]; \cP(\Td))$, let $\delta\mu \in D([0,T]^2)$ be the function
\[ [0,1]^2 \ni (s,t) \mapsto  \delta\mu (s,t) := d(\mu_s, \mu_t) \]
and notice that $\mu \mapsto \delta \mu$ is a continuous map. Moreover, to estimate the modulus of continuity of $\mu \in C([0,1]; \cP(\Td))$, it is equivalent to bound that of $\delta \mu \in C([0,1]^2; [0,\infty) )$, for given $s$, $t \in [0,1$, we use the inequality
\[ d(\mu_s, \mu_t) = | \delta \mu (s,t) - \delta \mu (t,t)|,\]
hence if $C$ is a modulus of continuity for $\delta \mu$, it is also a modulus of continuity for $\mu$. For a converse, we notice more generally that if $\gamma \in D([0,1]^2)$ and $C: [0,\infty) \to [0,\infty)$ is non-decreasing function such that $\gamma(s,t) \le C(|s-t|)$ for every $s$, $t \in [0,1]$, then for $s_1$, $s_2$, $t_1$, $t_2 \in [0,T]$, one has by the triangular inequality
\begin{equation*}\begin{split}|\gamma (s_1, t_1) - \gamma (s_2, t_2) | & \le |\gamma(s_1, t_1) - \gamma(s_2, t_1) | +  |\gamma (s_2, t_1) -\gamma(s_2, t_2) |\\
& \le C( |s_1 - s_2| ) + C( |t_1 - t_2| ) \le 2C( \sqrt{ |s_1 - s_2|^2 + |t_1-t_2|^2}) 
\end{split}\end{equation*}
hence $2C$ is a modulus of continuity for $\gamma$.

We next show the existence of coercive functionals $\psi_1, \psi_2:  D([0,1]^2) \to [0,\infty]$, so then by defining the coercive functional
\[  D([0,1]^2) \ni \gamma \mapsto \psi(\gamma) = \inf_{\gamma \le \gamma^1 + \gamma^2} \cur{ \psi_1(\gamma^1) + \psi_2(\gamma^2)},\]
we obtain our thesis choosing $\Psi(\mu):= \psi( \delta \mu )$.

For $\veps >0$, we let $\delta_1 = \delta_{1,\veps}$ be the largest number in the form $\delta = 1/n$, with $n \in \Nat$, such that $\delta^{h_i-1}< \veps^{2c_1}$, and we let $\delta_2 = \delta_{2,\veps}$ be the largest number $\delta =1/n$, with $n \in \Nat$, such that $\delta^{c_2/2-1}< \veps^{2c_2}$. We notice that these definition are well-posed because $c_1 \ge 2 >1$ and $c_2 >2$; they are a posteriori justified by the requirements~\eqref{eq:c1} and~\eqref{eq:c2}. Then, we introduce the closed sets
\begin{equation}\label{eq:set-modulus-contuniuity} A_i(\veps) :=  \Big\{ \gamma \in D([0,T]^2) \, : \, \sup_{k =1, \ldots, n_i} \, \sup_{s \in \sqa{(k-1) \delta_i, k\delta_i}}  \gamma(s, (k-1) \delta_i) \le \veps \Big\},\end{equation}
and we let $\psi_{i} (\gamma) := \sum_{m\ge 0} (m +1) \chi_{ A_i^c(2^{-m})}$, hence $\psi_{i}$ is lower semicontinuous, and $\psi_i(\gamma) \le m$ implies $\gamma \in A_i(2^{-k})$, for every $k \ge m$. Coercivity follows from the remark above, since  if we define the non-increasing function
\begin{equation*} C_{i,m}(x) := \left\{ \begin{array}{ll}
         2^{1-k} & \text{if $x \in [\delta_{i,2^{-(k+1)}}, \delta_{i,2^{-k}})$ with $k \ge m$,}\\
         2^{1-m}/ \delta_{i,2^{-m}} & \text{if $x \in [\delta_{i,2^{-m}}, +\infty)$},\end{array} \right. \end{equation*}
         then one has, for every $\gamma \in D([0,1]^2)$ with $\psi_i(\gamma) \le m$, $\gamma(s,t) \le C_{i,m}(|s-t|)$ for every $s$, $t \in [0,1]$, hence $2C_{i,m}$ is a modulus of continuity for $\gamma$.
             
To show that~\eqref{eq:general-tightness} holds, we may assume that the right hand side therein is finite, otherwise the thesis is trivial. By~\eqref{eq:wasserstein-coupling}, we estimate from above using the triangular inequality for the Euclidean norm on $\R^{n\times d}$,
\[ d(\mu_s, \mu_t) \le \bra{ \frac 1 n \sum_{i=1}^n \abs{ \int_s^t h_r(X^i_r)dr }^2}^{1/2}  + \bra{ \frac 1 n \sum_{i=1}^n \abs{ \int_s^t \sigma_r(X^i_r)dB^i_r }^2 }^{1/2},\]
and we let, for $s$, $t \in [0,1]$,
\[ \gamma^1 (s,t) := \bra{ \frac 1 n \sum_{i=1}^n \abs{ \int_s^t h_r(X^i_r)dr }^2}^{1/2} \quad \text{and} \quad \gamma^2 (s,t) := \bra{ \frac 1 n \sum_{i=1}^n \abs{ \int_s^t \sigma_r(X^i_r)dB^i_r }^2 }^{1/2}\]
hence
\[ \E\sqa{ \Psi(\mu) }  = \E\sqa{\psi( \delta \mu )} \le\E\sqa{ \psi_1( \gamma^1)  } + \E\sqa{\psi_2( \gamma^2)}.\]
For $i \in \{1,2\}$, we have
\[\E\sqa{\psi_i(\gamma^i)} = \sum_{m \ge 0} (m+1) \E\sqa{\chi_{ A_i^c(2^{-m})}  \circ \gamma^i}.  \]
We focus on each term of the series above, writing for brevity $\veps$ in place of $2^{-m}$. By~\eqref{eq:set-modulus-contuniuity}, we have
\begin{equation*}\E\sqa{\chi_{ A_i^c(\veps)}  \circ \gamma^i}  = P\Big( \sup_{k =1, \ldots, n_i} \, (\gamma^i)^*_k > \veps \Big) \le \sum_{k=1}^{n_i} P\bra{ (\gamma^i)^*_k> \veps},\end{equation*}
where we write, $(\gamma^i)^*_k := \sup_{s \in \sqa{(k-1) \delta_i, k\delta_i}} \gamma^i(s,(k-1) \delta_i)$.

For $i=1$, we estimate (using Jensen inequality and the assumption $c_1\ge 2$)
\begin{equation*}\begin{split} P\bra{ (\gamma^1)^*_k > \veps} & \le \frac{1}{\veps^{c_1}} \E\sqa{ \bra{ \frac 1 n \sum_{i=1}^n  \abs{\int_{(k-1)\delta_1}^{k\delta_1}  \abs{h_r(X^i_r) } dr}^2  }^{c_1/2}}\\
 &\le \frac{1}{\veps^{c_1} }\E\sqa{ \frac 1 n \sum_{i=1}^n \bra{ \int_{(k-1)\delta_1}^{k\delta_1}  \abs{h_r(X^i_r) } dr}^{c_1}  }\\
  &\le \frac{\delta_1^{c_1-1}}{\veps^{c_1}} \E\sqa{  \int_{(k-1)\delta_1}^{k\delta_1}  \int \abs{h_r }^{c_1} d\mu_r dr  }
\end{split}\end{equation*}
Summing upon $k \in \cur{1, \dots, \delta^{-1}_{1,\veps}}$, and $\veps = 2^{-m}$, for $m \ge 0$ we obtain
\[ \E\sqa{\psi_1 \circ \gamma^1} \le  a^1  \E\sqa{ \int_0^1 \int \abs{h_s}^{c_1} d\mu_s ds }\]
where
\begin{equation}\label{eq:c1}
 a^1 = \sum_{m\ge 0} (m+1) \delta_{1,2^{-m} }^{c_1-1} 2^{c_1 m} \le  \sum_{m\ge 0} (m+1)  2^{- m}< \infty,
\end{equation}

For $i=2$, at fixed $\veps = 2^{-m}$ and $k \in \cur{1, \dots, \delta^{-1}_{2,\veps}}$, we estimate similarly, using~\eqref{eq:bg-hilbert} for $p =c_2$,
\begin{equation*}\begin{split} P\bra{ (\gamma^2)^*_k > \veps} & \le \frac{1}{\veps^{c_2}} \E\sqa{ \sup_{s \in[(k-1) \delta_2, k \delta_2]} \bra{\frac 1 n \sum_{i=1}^n \abs{ \int_{(k-1) \delta_2} ^s \sigma_r(X^i_r)dB^i_r }^2 }^{c_2/2}} \\
& \le \frac{c_{c_2} }{\veps^{c_2}} \E\sqa{ \bra{ \int_{(k-1) \delta_2} ^{k \delta_2} \frac 1 n \sum_{i=1}^n |\sigma_r|^2(X^i_r)dr }^{c_2/2}}\\
 &\le \frac{\delta_2^{c_2/2 - 1}}{\veps^{c_2} }\E\sqa{ \int_{(k-1) \delta_2} ^{k \delta_2} \frac 1 n \sum_{i=1}^n |\sigma_r|^{c_2}(X^i_r)dr }\\
  &= \frac{\delta_1^{c_2/2-1}}{\veps^{c_2}} \E\sqa{  \int_{(k-1)\delta_1}^{k\delta_1}  \int \abs{\sigma_r }^{c_2} d\mu_r dr  }
\end{split}\end{equation*}
Summing upon $k \in \cur{0, \ldots, \delta^{-1}_{2,\veps}}$ and $\veps= 2^{-m}$, for $m \ge 0$, we obtain
\[ \E\sqa{\psi_2\circ \gamma^2} \le  a^2  \E\sqa{ \int_0^1 \int \abs{\sigma_s}^{c_2} d\mu_s ds }\]
where
\begin{equation}\label{eq:c2} a^2 = \sum_{m\ge 0} (m+1) \delta_{2,2^{-m} }^{c_2/2-1} 2^{c_2 m}  \le \sum_{m\ge 0} (m+1) 2^{-m} < \infty,\end{equation}
and the thesis follows.
\end{proof}

\begin{proposition}\label{prop:technical-narrow-convergence}
Let $(X,d)$, $(Y, \delta)$ be Polish metric spaces, let $(f^n)_{n\ge 1 }$ be a sequence of random variables with values in $X$, converging in law towards $f$ and let $(F^n)_{n \ge 1}$ be a sequence of maps $F_n: X \to Y$, pointwise converging towards $F: X \to Y$, with $\sup_{n\ge 1} \operatorname{Lip}F^n  := L < \infty$. Then, the sequence $(F^n \circ f^n)_{n\ge 1}$ converges in law towards $F \circ f$.
\end{proposition}

\begin{proof}
Let $p^n \in \cP(X)$ denote the law of $f^n$ and $q^n = (F^n)_\sharp p^n \in cP(Y)$ denote the law of $F^n \circ f^n$, and let $p$, $q$ denote respectively the laws of $f$ and $F \circ f$. Given $\varphi \in C_b(Y)$, we estimate
\begin{equation*}\begin{split}
\abs{ \int_Y \varphi q^n - \int_Y \varphi q } & = \abs{ \int_Y \varphi( F^n)  p^n - \int_Y \varphi(F) p }  \\
& \le \abs{ \int_Y \varphi(F^n) p^n - \int_Y \varphi(F) p^n } +  \abs{ \int_Y \varphi(F) p^n - \int_Y \varphi(F) p }.
\end{split}\end{equation*}
The latter term in the right hand side above converges to $0$, by definition of convergence in law. To estimate the former, we combine the facts that $(p^n)_{n\ge 1}$ is a tight family and that $\varphi(F^n) \to\varphi(F)$ uniformly on compact sets: given $\varepsilon >0$ and a compact set $K \subseteq X$ with $p^n(X) \ge 1-\veps$, for $n \ge 1$, we estimate
\[ \int_Y \abs{ \varphi(F^n)- \varphi(F) }p^n  \le \norm{\varphi (F^n) - \varphi (F) }_{L^\infty(K)} (1-\veps) + \norm{\varphi}_{L^\infty(Y) } \veps,\]
and we let first $n \to \infty$ and then $\veps\downarrow 0$.

To show uniform convergence on compact sets of $\varphi(F^n)$  towards $\varphi(F)$ we notice first that, given any compact set $K \subseteq X$, the set $\bigcup_{n\ge 1} F^n(K) \subseteq Y$ is pre-compact (i.e., its closure is compact). Indeed, given any sequence $(y^k)_{k \ge 1}$ in such a set, we may choose a corresponding sequence of points $(x^k)_{k\ge 1} \subseteq K$ and $n(k) \ge 1$  such that $F^{n(k)} (x^k) = y^k$ for every $k \ge 1$. We may always assume, up to extracting a subsequence, that $x^k \to x \in K$ as $k \to \infty$. If $n(k)$ is bounded, then for some $n \ge 1$ we have $n(k) = n$ for infinitely many $k$, so up to extracting a subsequence we have $y^k = F^n(x^k) \to F^n(x)$. Otherwise,  again up to a subsequence, we may assume that $n(k) \to \infty$, so that
\begin{equation*}\begin{split}
 \delta(F^{n(k)}(x^k), F(x) ) & \le \delta(F^{n(k)}(x^k), F^{n(k) } (x) ) +  \delta(F^{n(k)}(x), F(x) ) \\ 
 & \le L d(x^k, x ) + \delta(F^{n(k)}(x), F(x) ) \to 0
\end{split} \end{equation*}
as $k \to \infty$. As a consequence, $\varphi$ restricted to the closure of $\bigcup_{n\ge 1} F^n(K)$ is uniformly continuous, with a modulus of continuity $\omega$. Moreover, the family of maps $\varphi(F^n)$ is uniformly continuous on $K$ (uniformly in $n \ge 1$) since for $x^1$, $x^2 \in K$, 
\[ |\varphi(F^n(x^1)) -  \varphi(F^n(x^2))|  \le \omega\bra{ \delta(F^n(x^1), F^n(x^2) ) } \le \omega\bra{ L d(x^1, x^2) }. \]
From this it is straightforward that $\varphi(F^n) \to \varphi(F)$ uniformly on $K$: given $\veps>0$, we let $\alpha >0$ be such that $\omega{L \alpha } < \veps$ and consider a finite covering of $K$ with balls of radius $\alpha>0$ and centers $x^1 \ldots, x^k \in K$, so that for any $x \in K$, there is some $i \in \cur{1, \ldots k}$ such that
\begin{equation*}\begin{split}
|\varphi(F^n(x)) -  \varphi(F(x))|  & \le |\varphi(F^n(x)) -  \varphi(F^n(x^i))| + |\varphi(F^n(x^i)) -  \varphi(F(x^i))| + |\varphi(F(x^i)) -  \varphi(F^n(x))|\\
& < 2 \veps + |\varphi(F^n(x^i)) -  \varphi(F(x^i))| < 3\veps
\end{split} \end{equation*}
if $n$ is chosen sufficiently large.

%
%
\end{proof}

\begin{proposition}[weak-strong convergence]\label{prop:weak-strong}
Let $p \in [1, \infty)$ and let $(\mu^n)_{n\ge 1}$ be a sequence of random variables with values in $C([0,T]; \Td)$ converging in law towards  $\mu$, let $(\w^n)_{n\ge 1}$ be a sequence of mollifiers on $\Td$, with $\w^n \to \delta_0$, and set $\tilde{\mu}^n_t = \mu^n_t \conv \w^n$. Assume that $\tilde{\mu}^n_t = \rho_t^n(x) dx$, for a.e.\ $t \in [0,T]$, with $(\nabla \rho^n)_{n\ge 1}$ uniformly bounded in $L^p(\Omega \times [0,T] \times \Td  )$, i.e.,
\[ \sup_{n\ge 1} \E\sqa{ \int_0^T \int \abs{\nabla \rho^n_t }^p dt } < \infty.\]
Then, $\mu_t = \rho_t(x) dx$, for a.e.\ $t\in [0,T]$, with $\rho \in L^p(\Omega \times [0,T] \times\Td)$, and for every bounded continuous function $F: C([0,T]; \cP(\Td) ) \times L^p([0,T]; L^p(\Td)) \to \R$, with $F(\nu, \cdot)$ uniformly continuous (with respect to norm convergence) in $L^p([0,T]; L^p(\Td))$, uniformly in $\nu \in C([0,T]; \cP(\Td) )$, one has
\begin{equation}\label{eq:weak-strong}
\lim_{n \to \infty}\E\sqa{F(\mu^n, \rho^n)} = \E\sqa{ F(\mu, \rho) }.
\end{equation}
If $p\in (1, \infty)$, one has moreover $\rho \in L^p(\Omega \times [0,T], W^{1,p}(\Td))$.
\end{proposition}

In explicit terms, the uniform continuity assumption on $F(\nu, \cdot)$ means that there exists some modulus of continuity $\delta: [0,\infty) \to [0,\infty)$ such that, for every $\nu \in C([0,T]; \cP(\Td) )$, one has
\[
\abs{F(\nu, \rho^1) - F(\nu, \rho^2)} \le \delta\bra{ \norm{\rho^1- \rho^2}_{L^p_t(L^p_x)} },\quad  \text{ for every $\rho^1$, $\rho^2 \in L^p([0,T]; L^p(\Td))$.}
\]
Notice that, since $F$ is uniformly bounded, we may assume $\delta$ to be bounded as well.

\begin{proof}
Without any loss of generality, we consider the case $T=1$ only. Let us first notice that $\tilde{\mu} \to \mu$ in law as random variables with values in $C([0,1]; \cP(\Td))$, by Proposition~\ref{prop:technical-narrow-convergence}, using the fact that, for every $n \ge 1$, $\nu \mapsto (\nu_t \conv \w^n)_{t \in [0,1]}$ is a contraction with respect to the natural distance on $C([0,1]; \cP(\Td))$:
\[ d_1( \nu_t^1 \conv \w^n, \nu_t^1 \conv \w^n) \le  d_1( \nu_t^1,  \nu_t^1), \text{for every $t \in [0,T]$, $\nu^1, \nu^2 \in C([0,1]; \cP(\Td))$,}\]
and $d_1$ denotes the Wasserstein-Kantorovich distance with exponent $1$.

The second statement then follows from the fact that, for $p \in (1,\infty)$, the norm in $L^p([0,1]; W^{1,p}(\Td))$ is lower semicontinuous as a functional on $C([0,1]; \Td)$, when defined $+\infty$ outside of $L^p([0,1]; W^{1,p}(\Td))$ (for $p=1$, we would obtain a BV estimate).  Hence,
\[ \E\sqa{\norm{\mu}_{L^p_t(W^{1,p}_x) }^2} \le \liminf_{n \to \infty} \E\sqa{\norm{\tilde{\mu}^n}_{L^p_t(W^{1,p}_x) }^p}< \infty.\]
\newcommand{\sfP}{\mathsf{P}}
\newcommand{\sfp}{\mathsf{p}}
To prove the second statement, we use the smoothing action on measures  of the   standard heat semigroup on $\Td$, $(\sfP^\veps)_{\veps>0}$, i.e., the symmetric Markov transition semigroup associated to the Brownian motion on $\Td$.
For $\veps>0$, $\nu \in C([0,1]; \cP(\Td))$, we consider its action on space variables only, i.e., we let $(\sfP^\veps \nu)_t = \sfP^\veps \nu_t$. Since
 the convolution kernel $\sfp^\veps$ of $\sfP^\veps$ is smooth with gradient uniformly bounded by (some constant times) $1/\sqrt{\veps}$,  we have for $\nu^1$, $\nu^2 \in C([0,1]; \cP(\Td))$ the bound
\begin{equation*} \begin{split} \norm{ \sfP^\veps \nu^1 - \sfP^\veps \nu^2 }_{L^2_t(L^2_x) } &  \le \norm{ \sfP^\veps \nu^1 - \sfP^\veps \nu^2 }_{L^\infty_t(L^\infty_x) } \\ & \le  \sup_{t \in [0,T]} \sup_{x \in \Td} \abs{ \int \sfp^\veps(x-y) \nu^1_t(dy) - \int \sfp^\veps(x-y) \nu^2_t(dy) }\\
& \le \norm{ \nabla \sfp^\veps }_{L^\infty(\Td)}  \sup_{t \in [0,T]} d_1( \nu^1_t, \nu^2_t) \le \frac{c}{\sqrt{\veps}} d(\nu^1, \nu^2). 
\end{split}\end{equation*}
Hence for fixed $\veps>0$, one has
\[ \lim_{n \to \infty } \E\sqa{ F(\mu^n, \sfP^\veps \mu^n ) } \to \E\sqa{ F(\mu, \sfP^\veps \mu ) }.\]
Actually, by Proposition~\ref{prop:technical-narrow-convergence}, we also have
\[ \lim_{n \to \infty } \E\sqa{ F(\mu^n, \sfP^\veps \rho^n ) } \to \E\sqa{ F(\mu, \sfP^\veps\mu ) }.\]
On the other side, by the Poincar\'e inequality with exponent $p \in [1, \infty)$, for every $n \ge 1$ (as well as in the limit as $n \to \infty$),
\[
  \norm{  \sfP^\veps \rho^n - \rho^n }_{L^p_t(L^p_x) } \le c_p \sqrt{\veps} \norm{  \nabla \rho^n }_{L^p_t(L^p_x) },
\]
where $c_p$ is some constant depending on $p$ (and $d$) only. Hence,
\begin{equation*}\begin{split}
 \abs{ \E\sqa{ F(\mu^n, \sfP^\veps \rho^n )  } - \E\sqa{ F(\mu^n, \rho^n) } }  & \le  \E\sqa{ \delta\bra{ \norm{ \sfP^\veps \rho^n - \rho^n }_{L^2_t(L^2_x) } } }\\
  & \le  \E\sqa{ \delta\bra{ c_p \sqrt{\veps} \norm{ \nabla \rho^n }_{L^p_t(L^p_x) }  } } \\
    & \le   \delta\bra{c_p \sqrt{\veps} R }  +  \norm{\delta}_\infty \Prob\bra{ \norm{ \nabla \rho^n }_{L^p_t(L^p_x) } > R },
\end{split} \end{equation*}
where $\delta$ denotes a modulus of continuity for $F(\nu, \cdot)$, uniform with respect to $\nu \in C([0,T]; \cP(\Td))$. A similar inequality holds for $\mu$, $\rho$ in place of $\mu^n$ and $\rho^n$. Finally~\eqref{eq:weak-strong} follows the inequality
\begin{equation*}\begin{split}
\limsup_{n \to \infty} \abs{ \E\sqa{ F(\mu^n, \rho^n) - \E\sqa{ F(\mu, \rho)  }} }  & \le \limsup_{n \to \infty}  \abs{ \E\sqa{ F(\mu^n, \rho^n)} - \E\sqa{ F(\mu^n, \sfP^\veps \mu^n)  } } + \\
 & + \limsup_{n \to \infty} \abs{ \E\sqa{ F(\mu^n, \sfP^\veps \rho^n)} - \E\sqa{ F(\mu, \sfP^\veps \rho)  } }+\\
 & + \abs{ \E\sqa{ F(\mu, \rho)} - \E\sqa{ F(\mu, \sfP^\veps \rho )  } }\\
  & \le  2 \delta\bra{c_p \sqrt{\veps} R }  +  2 \frac{ \norm{\delta}_\infty}{R^p} \sup_{n \ge 1} \E\sqa{ \norm{ \nabla \rho^n }_{L^p_tL^p_x}^p },
\end{split}\end{equation*}
letting first $\veps \downarrow 0$ and then $R \uparrow \infty$.
\end{proof}

\bibliography{a-particle-system}

\end{document}